\documentclass[a4paper]{article}
\usepackage{latexsym,anysize}
\usepackage{amsmath,amssymb}
\usepackage{amsthm}
\allowdisplaybreaks[4]
\usepackage[english]{babel}
\usepackage[utf8x]{inputenc}
\usepackage{amsmath}
\usepackage{graphicx}
\usepackage{booktabs}
\usepackage{url}
\usepackage[colorinlistoftodos]{todonotes}
\usepackage[margin=1.2in]{geometry}
\usepackage{indentfirst}
\usepackage{hyperref}
\usepackage{makecell}
\usepackage{subfigure}
\usepackage{blindtext}
\usepackage{amsmath}
\usepackage{paralist}
\usepackage[misc]{ifsym}
\usepackage[misc]{ifsym}
\usepackage{epsfig} 
\usepackage{epstopdf} 
\usepackage{cite}
\newtheorem{theorem}{Theorem}[section]

\newtheorem{lemma}[theorem]{Lemma}

\newtheorem{definition}[theorem]{Definition}
\newtheorem{remark}[theorem]{Remark}

\title{\textbf{On the solutions to variable-order fractional $p$-Laplacian evolution equation with $L^1$-data}}
\author{\textbf{Sixuan Liu, Gang Dong, Hui Bi and Boying Wu}}
\date{}
\begin{document}
\maketitle
\large
\noindent\textbf{Abstract:}
This study investigates Dirichlet boundary condition related to a class of nonlinear parabolic problem with nonnegative $L^1$-data, which has a variable-order fractional $p$-Laplacian operator.
The existence and uniqueness of renormalized solutions and entropy solutions to the equation is proved. 
To address the significant challenges encountered during this process, we use approximation and energy methods.
In the process of proving, the well-posedness of weak solutions to the problem has been established initially, while also establishing a comparative result of solutions.\\
\textbf{keywords:}Fractional $p$-Laplacian, variable-order, renormalized solutions, entropy solutions, existence and uniqueness.\\

\section{Introduction}
Assume that $\Omega\subset\mathbb{R}^N$ is a bounded domain with a Lipschitz boundary $\partial\Omega$, where $N\geq2,\;T>0$, and $f$ is a measurable function. This article studies the following variable-order nonlinear parabolic problem:
\begin{equation}\label{Quotient}
    \left\{\begin{array}{ll}
    u_t+(-\Delta)_p^{s(\cdot,\cdot)}u=f\quad    &\text{in}\quad\Omega_T=\Omega\times(0,T),\\u=0\quad&\text{in}\quad(\mathbb{R}^N\backslash\Omega)\times(0,T),\\
    u(x,0)=u_0(x)\quad
    &\text{in}\quad\Omega,
    \end{array}\right.
\end{equation}
where $1<p<+\infty, 
s:\mathbb{R}^{N}\times\mathbb{R}^{N}\to(0,1)$ is a symmetric continuous function, 
$ps\left(x,y\right)<N$  holds for any $(x,y)\in\mathbb{R}^N\times\mathbb{R}^N$. 
And $(-\Delta)_p^{s(\cdot,\cdot)}$ is the variable-order fractional $p$-Laplacian operator defined as
\begin{equation*}
\begin{split}
    (-\Delta)_{p}^{s(\cdot,\cdot)}u(x,t)
    &:=\mathrm{P.V.}\int_{\mathbb{R}^{N}}
    \frac{|u(x,t)-u(y,t)|^{p-2}\big(u(x,t)-u(y,t)\big)}
    {|x-y|^{N+p{s(x,y)}}} dy\\
    &=\lim_{\varepsilon\to0^+}\int_{\mathbb{R}^{N}\setminus B_{\varepsilon}(x)}
    \frac{|u(x,t)-u(y,t)|^{p-2}\big(u(x,t)-u(y,t)\big)}
    {|x-y|^{N+p{s\left(x,y\right)}}} dy,
\end{split}
\end{equation*}
where $(x,t)\in\mathbb{R}^N\times\mathbb{R}^+$, 
P.V. is a commonly used abbreviation for “in the principal value sense”. It is also a normalization factor. Moreover, we assume that $f$ and $u_0$ are nonnegative functions.

It is commonly known that the study of the well-posedness of the fractional $p$-Laplacian equation with $L^1$ and measure data has been intensively investigated and applied\cite{ref25,ref24}.
In the examination of this category of nonlinear issues, variable-order spaces is of significant importance.
The mathematics team led by C. Zhang has conducted extensive research on fractional and variable exponents, thereby establishing a foundational framework for analyzing subsequent complex models, further advancing the development of this field in both mathematical theory and applications.
We consult \cite{ref100,ref101,ref102,ref103,ref104,ref105,ref106,ref107,ref108,ref109} for further information regarding the variable exponent Sobolev space and the associated equations that involve variable exponents.

The investigation of problems involving variable-order represents a novel and intriguing area of research.
The variable-order fractional $p$-Laplacian incorporates spatial dependence in the fractional-order operator denoted by $s$.
A variational model utilizing a variable-order function $x \to s(x)$ was introduced in \cite{ref4}. Subsequently, a definition of the operator $(-\Delta)^{s(x)}$ on $\mathbb{R}^N$ was established in \cite{ref18}, while the definition for bounded domains was presented in \cite{ref11}.
In a recent publication, Y.Li \cite{ref33} presented a compact embedding theorem along with some further related properties on the variable-order Sobolev fractional space.
Motivated by \cite{ref104} and  \cite{ref33}, we conduct a study focused on the renormalized solutions and entropy solutions to the variable-order fractional $p$-Laplacian equation with Dirichlet boundary condition and nonnegative $L^1$-data. 
It can be characterized as $\operatorname{div}\left(|\nabla u|^{p-2} \nabla u\right)$, which is associated with the variable-order Sobolev space.
To the best of our knowledge, there is currently no research addressing nonlinear evolution equations that involve variable-order and $L^1$-data.

In the mathematical research, for the case where $f \in L^2(\Omega)$, it is possible to demonstrate the well-posedness of weak solutions to the problem \eqref{Quotient} within a suitable Sobolev space. These solutions are shown to satisfy the equations in the sense of distributions.
Nonetheless, several studies concentrate on $L^1$-data. For integrable data \cite{ref16}, there may not exist such a weak solution due to less regularity.
For the problem \eqref{Quotient}, where $f \in L^1(\Omega)$ \cite{ref23}, it is necessary to define the "solution" in a weak sense after some form of renormalization, thereby defining the renormalized solution.
The notion of renormalized solutions was first presented by DiPerna and Lions in their study \cite{ref30} of the Boltzmann equation.
It is well known that the existence and uniqueness of the renormalized solutions were established by Alibaud in \cite{ref28}, and can be demonstrated through the following questions:
\begin{equation*}
    \beta(u)+(-\Delta)^{s} u \ni f \text { in } \mathbb{R}^{N},
\end{equation*}
where $f\in L^1(\mathbb{R}^{N})$ and $\beta$ is a maximal monotone graph in $\mathbb{R}^{N}$.
Lions et al.\cite{ref29} applied this concept to the nonlinear elliptic problems with Dirichlet boundary conditions.
In \cite{ref17} such a notion has been extended to the nonlinear parabolic problems even with variable exponent Laplacian term \cite{ref26} and a reaction-diffusion system \cite{ref34}.
Meanwhile, the concept of entropy solutions was introduced in \cite{ref5} for the nonlinear elliptic problems.
Abdellaoui et al. investigated the existence and uniqueness of nonnegative entropy solutions in \cite{ref6} for fractional $p$-Laplacian equations with weight and general data.
Additionally, it mentioned the existence of approximate limit solutions (distribution solutions) and weak solutions \cite{ref2,ref21}.
At the same time, the study of entropy solutions was also adapted to general nonlinear elliptic problems by Leone and Porretta in \cite{ref1,ref9}.
Renormalized (entropy) solutions is more suitable than weak solutions for working with nonlinear equations and measuring data due to less regularity.

This article aims to broaden the findings presented in \cite{ref33} to a class of parabolic equations, and provide theoretical analysis.
The study first characterizes the definition and establishes the existence and uniqueness of weak solutions by using Nonlinear Semigroup Theory. Subsequently, we develop an approximate sequence of solutions and derive several priori estimates. Following this, the thesis describes a subsequence to attain a limit function and demonstrates that this function is a renormalized solution by the strong convergence of the truncations.
By selecting suitable test functions, we prove the uniqueness of renormalized and entropy solutions.
In contrast to earlier studies, our research contributes variable-order Laplacian studies for $L^1$-data, at the same time it is complicated for Dirichlet boundary conditions. The proof of the weak solutions is also unprecedented.

The paper is structured as follows. Section 2 reviews some definitions and fundamental properties of variable-order fractional Sobolev spaces. It also provides some notations and main results concerning the renormalized and entropy solutions to the problem \eqref{Quotient}. Section 3 presents the well-posedness of weak solutions for the approximate problem \eqref{Eqn4} associated with \eqref{Quotient}. Section 4 provides detailed proof of the main results, including existence, uniqueness and comparison principles. 

\section{Mathematical preliminaries and main result}

In this section, we recall some definitions and basic properties of variable-order fractional Sobolev spaces at first. 

Suppose that $\Omega\subset\mathbb{R}^N$ is a bounded domain, $1<p<+\infty,\;s\left(x,y\right):\mathbb{R}^{N}\times\mathbb{R}^{N}\to(0,1)$ is a symmetric continuous function satisfies
\begin{equation*}
    0<s^-=\inf_{(x,y)\in\mathbb{R}^N\times\mathbb{R}^N}s(x,y)\leq s(x,y)\leq s^+=\sup_{(x,y)\in\mathbb{R}^N\times\mathbb{R}^N}s(x,y)<1.
\end{equation*}
The Gagliardo seminorm of a measurable function $u$ in ${\mathbb{R}^{N}}$ with variable order follows
\begin{equation*}
    [u]_{W^{s(\cdot,\cdot),p}(\mathbb{R}^{N})}    =\left(\int_{\mathbb{R}^{N}}\int_{\mathbb{R}^{N}}\dfrac{|u(x)-u(y)|^p}
    {|x-y|^{N+p{s\left(x,y\right)}}}dxdy\right)
    ^{\frac{1}{p}},
\end{equation*}   
the variable-order fractional Sobolev-Slobodeckij spaces is defined as a Banach space
\begin{equation*}
    W^{s(\cdot,\cdot),p}(\mathbb{R}^{N})=\left\{u\in L^p(\mathbb{R}^{N}):\left(\int_{\mathbb{R}^{N}}\int_{\mathbb{R}^{N}}\dfrac{|u(x)-u(y)|^p}
    {|x-y|^{N+p{s\left(x,y\right)}}}dxdy\right)
    ^{\frac{1}{p}}<+\infty\right\},
\end{equation*}
characterized by a norm denoted as 
\begin{equation*}
    \|u\|_{W^{s(\cdot,\cdot),p}(\mathbb{R}^{N})}=\|u\|_{L^p(\mathbb{R}^{N})}+[u]_{W^{s(\cdot,\cdot),p}(\mathbb{R}^{N})}.
\end{equation*}
Consequently, $W^{s(\cdot,\cdot),p}(\Omega)$ is separable and reflexive.

Denote $D_\Omega=(\mathbb{R}^{N}\times\mathbb{R}^{N})\backslash(\Omega^c\times \Omega^c)$, where $\Omega^c=\mathbb{R}^N\backslash\Omega$.
The linear space $X^{s(\cdot,\cdot),p}(\Omega)$ of Lebesgue measurable functions $u:\mathbb{R}^N\to\mathbb{R}$ is denoted as
\begin{equation*}
    \left\{u\in L^p(\mathbb{R}^{N}):\left(\int_\Omega|u(x)|^p dx
    +\iint_{D_\Omega}
    \frac{|u(x)-u(y)|^p}{|x-y|^{N+p{s(x,y)}}}
    dxdy\right)^\frac{1}{p}<+\infty\right\}.
\end{equation*}
It is evident that both bounded and Lipschitz functions are included in the space $X^{s(\cdot,\cdot),p}(\Omega)$, therefore $X^{s(\cdot,\cdot),p}(\Omega)$ is not trivial, and
\begin{equation*}
    X_0^{s(\cdot,\cdot),p}(\Omega)=\left\{u\in X^{s(\cdot,\cdot),p}(\Omega):u=0\;a.e.\;\text{in}\;\Omega^c\right\}.
\end{equation*}

Next, for the convenience of the readers, we present several lemmas that can be proved by following the arguments given in Lemma 6.1 \cite{ref19}.
\begin{lemma}\label{L1}
Let $\Omega\subset\mathbb{R}^N$ be a bounded domain, and $m(\Omega)\subset\mathbb{R}^{N}$ be a measurable set with finite measure. Fix $x\in\mathbb{R}^{N}$, let $p>1,\;s(x,y)\in(0,1)$, there holds
\begin{equation*}
    \int_{\Omega^c}\frac{1}{|x-y|^{N+p{s\left(x,y\right)}}}dy
    \geq c\cdot m(\Omega)^{-\frac{p{s^+}}{N}}
\end{equation*} 
for a suitable constant $c=c(N, p, s^+)>0$. 
\end{lemma}

\begin{proof}
Observe that $\displaystyle\rho=\left(\frac{m(\Omega)}{\omega_n}\right)^{\frac{1}{n}}$ and 
$m\big(\Omega^c\cap B_{\rho}(x)\big)=m\big(\Omega\cap B_{\rho}^c(x)\big)$, then
\begin{equation*}
\begin{split}
    \int_{\Omega^c}\frac{1}{|x-y|^{N+p{s\left(x,y\right)}}}dy
    &=\int_{\Omega^c \cap B_{\rho}(x)}
    \frac{1}{|x-y|^{N+p{s\left(x,y\right)}}}dy
    +\int_{\Omega^c\cap B_{\rho}^c(x)}
    \frac{1}{|x-y|^{N+p{s\left(x,y\right)}}}dy\\
    &\geq\int_{\Omega^c\cap B_{\rho}(x)}
    \frac{1}{\rho^{N+p{s\left(x,y\right)}}}dy
    +\int_{\Omega^c\cap B_{\rho}^c(x)}
    \frac{1}{|x-y|^{N+p{s\left(x,y\right)}}}dy\\
    &\geq\frac{m(\Omega^c\cap B_{\rho}(x))}{\rho^{N+p{s^+}}}+
    \int_{\Omega^c\cap B_{\rho}^c(x)}
    \frac{1}{|x-y|^{N+p{s^+}}}dy\\
    &=\frac{m(\Omega\cap B_{\rho}^c(x))}{\rho^{N+p{s^+}}}
    +\int_{\Omega^c\cap B_{\rho}^c(x)}
    \frac{1}{|x-y|^{N+p{s^+}}}dy\\
    &\geq\int_{\Omega\cap B_{\rho}^c(x)}
    \frac{1}{|x-y|^{N+p{s^+}}}dy
    +\int_{\Omega^c\cap B_{\rho}^c(x)}
    \frac{1}{|x-y|^{N+p{s^+}}}dy\\
    &=\int_{B_{\rho}^c(x)}\frac{1}{|x-y|^{N+p{s^+}}}dy. 
\end{split}
\end{equation*}
Therefore 
\begin{equation*}
    \int_{\Omega^c}\frac{1}{|x-y|^{N+p{s\left(x,y\right)}}}dy
    \geq c\cdot m(\Omega)^{-\frac{p{s^+}}{N}}.
\end{equation*}
\end{proof}

\begin{lemma}\label{L2}
Let $\Omega\subset\mathbb{R}^N$ be a bounded domain, for every function $u\in X_{0}^{s(\cdot,\cdot),p}(\Omega)$, then
\begin{equation*}
    \int_\Omega|u(x)|^p dx
    \leq C\iint_{D_\Omega}\frac{|u(x)-u(y)|^p}{|x-y|^{N+p{s\left(x,y\right)}}} dxdy
\end{equation*} 
holds for a positive constant $C=C(N,p,s^+,\Omega)$. 
\end{lemma}

\begin{proof}
\begin{equation*}
\begin{split}
       &\int_{\mathbb{R}^{N}}\int_{\mathbb{R}^{N}}
       \frac{|u(x)-u(y)|^{p}}{|x-y|^{N+p{s\left(x,y\right)}}} dxdy\\       &=\int_{\Omega\cup{\Omega^c}}\left(\int_{\Omega}\frac{|u(x)-u(y)|^{p}}{|x-y|^{N+p{s\left(x,y\right)}}}dy
       +\int_{\Omega^c}\frac{|u(x)-u(y)|^{p}}{|x-y|^{N+p{s\left(x,y\right)}}}dy\right)dx\\
       &=\int_{\Omega}\int_{\Omega}\frac{|u(x)-u(y)|^{p}}{|x-y|^{N+p{s\left(x,y\right)}}} dxdy       +2\int_{\Omega}|u(x)|^{p}dx\int_{\Omega^c}\frac{1}{|x-y|^{N+p{s\left(x,y\right)}}} dy.    
\end{split}
\end{equation*}
According to the Lemma \ref{L1}, we have
\begin{equation*}
    \int_{\mathbb{R}^{N}}\int_{\mathbb{R}^{N}}
    \frac{|u(x)-u(y)|^{p}}{|x-y|^{N+p{s\left(x,y\right)}}} dxdy
    -\int_{\Omega}\int_{\Omega}\frac{|u(x)-u(y)|^{p}}{|x-y|^{N+p{s\left(x,y\right)}}} dxdy     
    \geq 2c\int_\Omega|u(x)|^p dx
\end{equation*}
i.e. for any $p>1$, there exists a positive constant $C=C(N,p,s^+,\Omega)$
\begin{equation}\label{Eqn2}
    \int_\Omega|u(x)|^p dx
    \leq C\iint_{D_\Omega}\frac{|u(x)-u(y)|^p}{|x-y|^{N+p{s\left(x,y\right)}}} dxdy.
\end{equation}
\end{proof}
Therefore, for any $u\in X_{0}^{s(\cdot,\cdot),p}(\Omega)$
\begin{equation*}
    \iint_{D_\Omega}\frac{|u(x)-u(y)|^p}{|x-y|^{N+p{s\left(x,y\right)}}} dxdy
    \leq\|u\|_{W^{s(\cdot,\cdot),p}(\mathbb{R}^{N})}^p
    \leq C\iint_{D_\Omega}\frac{|u(x)-u(y)|^p}{|x-y|^{N+p{s\left(x,y\right)}}} dxdy, 
\end{equation*}
we can endow $X_{0}^{s(\cdot,\cdot),p}(\Omega)$ with the equivalent norm
\begin{equation*}
    \left\|u\right\|_{X_0^{s(\cdot,\cdot),p}(\Omega)}=\left(\iint_{D_\Omega}\frac{|u(x)-u(y)|^p}{|x-y|        ^{N+p{s\left(x,y\right)}}}dxdy\right)^{\frac{1}{p}}.
\end{equation*}
Note that $X_{0}^{s(\cdot,\cdot),p}(\Omega)$ is a uniformly convex and reflexive Banach space. 
The space $X_{0}^{s(\cdot,\cdot),p}(\Omega)$ is defined as the closure of $C_0^\infty(\Omega)$ with respect to the function space $X^{s(\cdot,\cdot),p}(\Omega)$. It is noted that $X_{0}^{s(\cdot,\cdot),p}(\Omega)$ is not equivalent to the standard fractional Sobolev space $W_{0}^{s(\cdot,\cdot),p}(\Omega)$.

For $u\in W^{s(\cdot,\cdot),p}(\mathbb{R}^N)$, we define the variable-order fractional $p$-Laplacian as
\begin{equation*}
    (-\Delta)_p^{s(\cdot,\cdot)}u(x)        :=\text{P.V.}\int_{\mathbb{R}^N}\frac{|u(x)-u(y)|^{p-2}
    \big(u(x)-u(y)\big)}{|x-y|^{N+p{s(x,y)}}}dy,
\end{equation*}
for all $u,v\in W^{s(\cdot,\cdot),p}(\mathbb{R}^N)$, then
\begin{equation*}
    \langle(-\Delta)_{p}^{s(\cdot,\cdot)}u,v\rangle
    =\frac{1}{2}\iint_{\mathbb{R}^{N}\times\mathbb{R}^{N}}\frac{|u(x)-u(y)|^{p-2}\big(u(x)-u(y)\big)}{|x-y|^{N+p{s(x,y)}}}\cdot\big(v(x)-v(y)\big)dxdy. 
\end{equation*}
Consequently, if $u,v\in X^{s(\cdot,\cdot),p}(\mathbb{R}^N)$
\begin{equation*}
    \langle(-\Delta)_{p}^{s(\cdot,\cdot)}u,v\rangle
    =\frac{1}{2}\iint_{D_\Omega}\frac{|u(x)-u(y)|^{p-2}\big(u(x)-u(y)\big)}{|x-y|^{N+p{s\left(x,y\right)}}}\cdot\big(v(x)-v(y)\big)dxdy. 
\end{equation*}
Herein $(-\Delta)_p^{s(\cdot,\cdot)}:X_0^{s(\cdot,\cdot),p}(\Omega)\to X_0^{s(\cdot,\cdot),p}(\Omega)^*$ where $X_0^{s(\cdot,\cdot),p}(\Omega)^*$ is the dual space of $X_0^{s(\cdot,\cdot),p}(\Omega)$. 
The corresponding parabolic space $L^p(0,T;X_0^{s(\cdot,\cdot),p}(\Omega))$ is defined as 
\begin{equation*}
    L^p(0,T;X_0^{s(\cdot,\cdot),p}(\Omega))=\left\{u\in L^{p}(\Omega_{T}):\|u\|_{L^{p}(0,T;X_{0}^{s(\cdot,\cdot),p}(\Omega))}<+\infty\right\}, 
\end{equation*}
with the norm
\begin{equation*}   
    \left\|u\right\|_{L^p(0,T;X_0^{s(\cdot,\cdot),p}(\Omega))}=\left(\int_0^T\iint_{D_\Omega}\frac{|u(x,t)-u(y,t)|^p}{|x-y|^{N+p{s\left(x,y\right)}}}dxdydt\right)^{\frac{1}{p}}.
\end{equation*}
$L^{p}(0,T;X_{0}^{s(\cdot,\cdot),p}(\Omega))$ is a Banach space whose dual space is $L^{p'}(0,T;X_0^{s(\cdot,\cdot),p}(\Omega)^*)$.

Let us denote
\begin{align*}
    &\widetilde{u}(x,y,t)=u(x,t)-u(y,t),\\
    &d\sigma=\frac{1}{|x-y|        ^{N+p{s\left(x,y\right)}}}dxdy.
\end{align*}
Define $T_k$ is the truncation function for $k\geq0$:
\begin{equation*}
    T_k(r)=\min\big\{k,\max\{r,-k\}\big\}=\begin{cases}k\quad&\text{if}\;r\geq k,\\r\quad&\text{if}\;|r|<k,\\-k\quad&\text{if}\;r\leq-k,\end{cases}  
\end{equation*}
and its primitive $\Theta_k:\mathbb{R}\to\mathbb{R}^+$ donated by
\begin{equation*}
    \Theta_k(r)=\int_0^rT_k(s) ds=\begin{cases} \displaystyle\frac{r^2}{2}\quad&\text{if}\;|r|\le k,\\ \displaystyle k|r|-\frac{k^2}{2}\quad&\text{if}\;|r|\ge k,\end{cases}
\end{equation*}
It is easy to obtain $0\leq\Theta_{k}(r)\leq k|r|$. 
Besides
\begin{equation*}
\begin{split}
    \mathcal{T}_{0}^{s(\cdot,\cdot), p}(\Omega_{T})
    &=\big\{u:\mathbb{R}^{N}\times(0,T]\to \mathbb{R}\;
    \text{is measurable and}\;\\
    &\qquad T_k(u)\in L^p(0,T;X_0^{s(\cdot,\cdot),p}(\Omega))\;
    \text{for every}\;k>0\big\}.
\end{split}
\end{equation*}  

Next, we provide the definitions of renormalized solutions and entropy solutions. 

\begin{definition}\label{D1} 
A function $u\in\mathcal{T}_0^{s(\cdot,\cdot),p}(\Omega_T)\cap C([0,T];L^1(\Omega))$ is said to be a renormalized solution to the problem \eqref{Quotient}, if the following conditions hold:\\
(1) For $D_{h}=\left\{(u,v)\in\mathbb{R}^{2}:\mathrm{min}\{|u|,|v|\}\leq h\;\text{and}\;\mathrm{max}\{|u|,|v|\}\geq h+1\;\text{or}\;uv<0\right\}$,  
\begin{equation*}
    \lim_{h\to\infty}\iiint_{\{(x,y,t):(u(x,t),u(y,t))\in D_{h}\}}
    |\widetilde{u}(x,y,t)|^{p-1} d\sigma dt=0. 
\end{equation*}            
(2) For every function $\phi\in C^1(\bar{\Omega}_T)$ with $\phi=0$ in
$\Omega^c\times(0,T)$ and $\phi(\cdot,T)=0$ in $\Omega$, 
and $H\in W^{1,\infty}(\mathbb{R})$ which is piecewise $C^1$ satisfying that $H'$ has a compact support
\begin{equation*}
\begin{split}
    &-\int_{0}^{T} \int_{\Omega}H(u)\frac{\partial\phi}{\partial t} dxdt 
    -\int_{\Omega}H(u_{0})\phi(x,0) dx\\
    &+\frac{1}{2}\int_{0}^{T}\int_{D_\Omega}|\widetilde{u}(x,y,t)|^{p-2}\widetilde{u}(x,y,t)\Big[\big(H^{\prime}(u)\phi\big)(x,t)-\big(H^{\prime}(u)\phi\big)(y,t)\Big] d\sigma dt\\
    &=\int_{0}^{T}\int_{\Omega}f\big(H^{\prime}(u)\phi\big) dxdt.                 
\end{split}
\end{equation*} 
\end{definition}

\begin{remark}
$(H'(u)\phi)(x,t)-(H'(u)\phi)(y,t)$ is symmetric, i.e.
\begin{equation*}
\begin{split}      
    \big(H'(u)\phi\big)(x,t)-\big(H'(u)\phi\big)(y,t)
        &=\big(H'(u)(x,t)-H'(u)(y,t)\big)\cdot\frac{\phi(x,t)+\phi(y,t)}2\\
        &+\frac{H'(u)(x,t)+H'(u)(y,t)}2\cdot\big(\phi(x,t)-\phi(y,t)\big).    
\end{split}  
\end{equation*}
\end{remark}

\begin{definition}\label{D2}  
A function $u\in\mathcal{T}_0^{s(\cdot,\cdot),p}(\Omega_T)\cap C([0,T];L^1(\Omega))$ is said to be an entropy solution to the problem \eqref{Quotient}, if $u$ satisfies:
\begin{equation*}
\begin{split}
    &\int_{\Omega}\Theta_{k}(u-\phi)(T) dx-\int_{\Omega}\Theta_{k}(u_{0}-\phi)(0) dx
    +\int_{0}^{T}\int_{\Omega}\phi_{t}T_{k}(u-\phi)dxdt\\
    &+\frac{1}{2}\int_{0}^{T}\int_{D_\Omega}|\widetilde{u}(x,y,t)|^{p-2}\widetilde{u}(x,y,t)\\
    &\cdot\big[\big(T_k(u)-\phi)(x,t)-\big(T_k(u)-\phi\big)(y,t)\big]d\sigma dt \\
    &\leq\int_0^T\int_\Omega fT_k(u-\phi) dxdt.    
\end{split}
\end{equation*} 
\end{definition}

\begin{theorem}\label{result1}
There exists a measurable function $u$ such that it is a unique renormalized solution to problem \eqref{Quotient}
in the sense of Definition \ref{D1} if $f \in L^1(\Omega_T)$ and $ u_0 \in L^1(\Omega)$. 
\end{theorem}

\begin{theorem}\label{result2}
There exists a measurable function $u$ such that it is a unique entropy solution to problem \eqref{Quotient}
in the sense of Definition \ref{D2} if $f \in L^1(\Omega_T)$ and $ u_0 \in L^1(\Omega)$. 
\end{theorem}

\begin{remark}
Renormalized solutions in Theorem \ref{result1} is equivalent to entropy solutions in Theorem \ref{result2}.
\end{remark}

\begin{theorem}\label{result3}
Assume $u_0\leq v_0$ and $f \leq g$, let $u_0,v_0\in L^1(\Omega), f,g\in L^1(\Omega_T)$. If $u$ is the renormalized (entropy) solution of problem \eqref{Quotient}, where $u_0, f$ 
is replace by $v_0, g$ and $v$ is also the renormalized (entropy) solution of problem \eqref{Quotient}, then $u \leq v$ a.e. in $\Omega_T$. It is also called the Comparison Principle.
\end{theorem}

\section{Auxiliary Theorem}

In this section, we introduce some notation and theorem used later that can be proved by following the arguments given in \cite{ref33,ref35}. In fact, fractional $p$-Laplace equations correspond to a certain completely accretive operator in $L^2$ space that can generate a compressed semigroup under specific range conditions and the following theorem can be derived.
\begin{theorem}\label{result4}
If $A$ is a $m$-completely accretive operator in $L^p(\Omega)$ for all $p > 1$, then the abstract Cauchy problem  
\begin{equation}\label{Eqn3}
\begin{cases}
    \displaystyle
    \frac{du}{dt}+Au=f\\
    u(0)=u_0
\end{cases}
\end{equation}
admits a unique mild solution $u$ for any initial datum $u_0 \in \overline{\mathrm{D}(A)}^{L^p(\Omega)}$ and any non-negative $f$. Moreover, if $A$ is the subdifferential of a proper lower semicontinuous and convex function in $L^2(\Omega)$ and $f\in L^2(0, T; L^2(\Omega))$, then the mild solutions of the above problem  \eqref{Eqn3} is also a strong solution. 
\end{theorem}

For all $p>1$, the approximate problem of \eqref{Quotient} defined as
\begin{equation}\label{Eqn4}
    \left\{\begin{array}{ll}(u_n)_t+(-\Delta)_p^{s(\cdot,\cdot)}u_n=f_n\quad &\text{in}\quad\Omega_T,\\u_n=0\quad&\text{in}\quad\Omega^c\times(0,T),\\u_n(x,0)=u_{0n}(x)\quad&\text{in}\quad\Omega.\end{array}\right.
\end{equation}

\begin{definition}\label{D3} 
Let $\Omega\subset\mathbb{R}^N$ be a bounded domain, $u_0\in L^2(\Omega)$ and $f_n\in L^2(0, T; L^2(\Omega))$, a function $u_n\in  W^{1,1}(0,T; L^2(\Omega))$ is called a weak solution to problem \eqref{Eqn4} for any $\varphi\in L^2(\Omega)\cap X_0^{s(\cdot,\cdot),p}(\Omega)$ and a.e. $t\in(0, T)$, the following identity holds:
\begin{equation*}
    \int_\Omega (u_n)_{t}\varphi dx+\frac{1}{2}\int_{D_\Omega}|\widetilde{u}_n(x,y)|^{p-2}\widetilde{u}_n(x,y)\big(\varphi(x)-\varphi(y)\big) d\sigma
    =\int_{\Omega}f_n\varphi dx.
\end{equation*}
\end{definition}

To study problem \eqref{Eqn4} from the perspective of Nonlinear Semigroup Theory, we define a specific operator $\mathcal{A}_{p}^{s(\cdot, \cdot)}: L^2(\Omega)\to[0,+\infty]$ that is closely associated with our problem, there holds
\begin{equation*}
    \mathcal{A}_{p}^{s(\cdot, \cdot)}\left(u_n\right)=\begin{cases}
    \displaystyle\frac{1}{2 p} \int_{D_\Omega}\left|\widetilde{u}_n(x,y)\right|^{p}d\sigma\quad&\text{if}\;u_n\in L^2(\Omega)\cap X_0^{s(\cdot,\cdot),p}(\Omega)\\
    +\infty\quad&\text{if}\;u_n\in L^2(\Omega)\backslash X_0^{s(\cdot,\cdot),p}(\Omega)
    \end{cases}
\end{equation*}

By Fatou's Lemma, it is easy to see that $\mathcal{A}_{p}^{s(\cdot,\cdot)}$ is weak lower semicontinuous and convex in $L^2(\Omega)$, 
then we have that the subdifferential $\partial\mathcal{A}_{p}^{s(\cdot,\cdot)}$ is a maximal monotone operator in $L^2(\Omega)$. 
Consequently, the following definition is presented regarding $\partial\mathcal{A}_{p}^{s(\cdot,\cdot)}$.

\begin{definition}\label{D4}
The operator $A_{p}^{s(\cdot, \cdot)}$ in $L^2(\Omega)\times L^2(\Omega)$ is denoted by $v_n\in A_p^{s(\cdot,\cdot)}(u_n)$, if $u_n, \varphi\in L^2(\Omega)\cap   X_0^{s(\cdot,\cdot),p}(\Omega)$, $v_n\in L^2(\Omega)$ and
\begin{equation*}
    \int_\Omega v_n\varphi dx +\frac{1}{2}\int_{D_\Omega}|\widetilde{u}_n(x,y)|^{p-2}\widetilde{u}_n(x,y)\big(\varphi(x)-\varphi(y)\big) d\sigma
    =0.
\end{equation*}
\end{definition}

Next, we prove the conditions that the operator $A_p^{s(\cdot,\cdot)}$ satisfied for the application of Nonlinear Semigroup Theory.

\begin{theorem} \label{result5}
The operator $A_{p}^{s(\cdot,\cdot)}$ is $m$-completely accretive in $L^2(\Omega)$ with dense domain and satisfied $A_{p}^{s(\cdot,\cdot)}=\partial\mathcal{A}_{p}^{s(\cdot,\cdot)}$.\\
\end{theorem}
\begin{proof}
We prove by several steps. 

\textbf{Step 1} Firstly, we let $u_{n1},\;u_{n2}\in  X_0^{s(\cdot,\cdot),p}(\Omega),\;v_{n1},\;v_{n2}\in A_p^{s(\cdot,\cdot)}(u_n),\;\gamma\in P_0$, that is, $\gamma\in C^\infty(\mathbb{R}),\;0\leq\gamma\leq1,\;\mathrm{supp}(\gamma')$ is compact, $0\not\in\mathrm{supp}(\gamma),\;\gamma(u_{n1}-u_{n2})=\varphi\in  X_0^{s(\cdot,\cdot),p}(\Omega)\cap L^{\infty}(\Omega)$. Then we can get
\begin{equation*}
\begin{split}
    &\int_{\Omega}\big(v_{n1}(x)-v_{n2}(x)\big)\gamma\big(u_{n1}(x)-u_{n2}(x)\big)dx\\
    &+\frac12\int_{D_\Omega}\Big(|\widetilde{u}_{n1}(x,y)|^{p-2}\widetilde{u}_{n1}(x,y)-|\widetilde{u}_{n2}(x,y)|^{p-2}\widetilde{u}_{n2}(x,y)\Big) \\
    &\cdot\Big[\gamma\big(u_{n1}(x)-u_{n2}(x)\big)-\gamma\big(u_{n1}(y)-u_{n2}(y)\big)\Big]d\sigma \\
    &=0,    
\end{split}
\end{equation*}
thus
\begin{equation*}
\begin{split}
    &\int_{\Omega}\big(v_{n1}(x)-v_{n2}(x)\big)\gamma\big(u_{n1}(x)-u_{n2}(x)\big)dx\\
    &\geq\frac12\int_{D_\Omega}\Big(|\widetilde{u}_{n2}(x,y)|^{p-2}\widetilde{u}_{n2}(x,y)-|\widetilde{u}_{n1}(x,y)|^{p-2}\widetilde{u}_{n1}(x,y)\Big) \\
    &\cdot\Big[\gamma\big(u_{n1}(x)-u_{n2}(x)\big)-\gamma\big(u_{n1}(y)-u_{n2}(y)\big)\Big]d\sigma.
\end{split}
\end{equation*}
The left side of the above formula is $v_n$, the right side $|\widetilde{u}_n(x,y)|^{p-2}\widetilde{u}_n(x,y)$, this process is monotonically increasing, it indicates that $A_p^{s(\cdot,\cdot)}$ is completely accretive. 

\textbf{Step 2} Given function $h(x)$ a.e. in $L^2(\Omega)$,  we consider the variational function
\begin{equation*}
    E(u)=\mathcal{A}_p^{s(\cdot,\cdot)}(u)+\frac{1}{2}\|u\|_{L^2(\Omega)}^2-\int_\Omega hudx,
\end{equation*}
for any $u\in L^2(\Omega)$. 
By Young’s inequality, we know that
\begin{equation*}
    E(u)\geq\frac{1}{2}\|u\|_{L^{2}(\Omega)}^{2}-\int_{\Omega}hudx\geq\frac{1}{4}\|u\|_{L^{2}(\Omega)}^{2}-\|h\|_{L^{2}(\Omega)}^{2}.
\end{equation*}
Thus $E(u)$ has a lower boundary and $\displaystyle\inf\limits_{u\in L^2(\Omega)}E(u)$ is the infimum. According to the definition of the infimum, there exists a minimalist sequence $u_k\in X_0^{s(\cdot,\cdot),p}
(\Omega)\cap L^2(\Omega)$ such that $\displaystyle\lim\limits_{k\to\infty}E(u_k)=\inf\limits_{u\in L^2(\Omega)}E(u)$ with $k\in(0,+\infty)$ i.e. $E(u_k)\leq C$ for a constant $C>0$. Therefore
\begin{equation*}
    \mathcal{A}_{p}^{s(\cdot,\cdot)}(u_{k})+\frac{1}{2}\|u_{k}\|_{L^{2}(\Omega)}^{2}-\int_{\Omega}hu_{k}dx\leq C,\;
\end{equation*}
i.e.
\begin{equation*}
    \mathcal{A}_{p}^{s\:(\cdot,\cdot)}(u_{k})+\frac{1}{4}\|u_k\|_{L^{2}(\Omega)}^{2}\leq C+\|h\|_{L^{2}(\Omega)}^{2}.
\end{equation*}

By Lemma 1 in \cite{ref33}, we derive that $u_k$ is bounded in $X_0^{s(\cdot,\cdot),p}(\Omega)$.
There exists a weak convergence subsequence $\displaystyle\{u_{k_i}\} (i=1,2\ldots)$ in minimization sequence $\{u_k\}$ that $u_{k_i} \rightharpoonup u_n$ weakly in $X_0^{s(\cdot,\cdot),p}(\Omega)$. Since the sequence $\{u_{k_i}\}$ is bounded in $L^2(\Omega)$, we can conclude that $u_n\in L^2(\Omega)$. Furthermore, $E(u)$ is a lower semicontinuous function by Theorem \ref{result4}, then
\begin{equation*}
    \inf\limits_{u\in L^2(\Omega)}E(u)\leq E(u_n)\leq\varliminf_{k\to\infty}E(u_{k_i})\leq\lim\limits_{k\to\infty}E(u_{k_i})=\inf\limits_{u\in L^2(\Omega)}E(u).
\end{equation*}

We conclude that $E(u_n)=\inf\limits_{u\in L^2(\Omega)}E(u)$, $u_n$ is the unique minimizer and satisfies the Euler-Lagrange equation because of the convexity. Then we have
\begin{equation*}
\begin{split}
    \left.\frac{d E(u_{n}+\varepsilon \varphi)}{d\varepsilon}\right|_{\varepsilon=0}
    &=\frac{1}{2}\int_{D_\Omega}|\widetilde{u}_{n}(x,y)|^{p-2}\widetilde{u}_n(x,y)\big(\varphi(x)-\varphi(y)\big)d\sigma\\
    &-\int_{\Omega}\big(h(x)-u_n(x)\big)\varphi(x)dx=0.
\end{split}
\end{equation*}
According to the Definition \ref{D4}, let $h-u_{n}=v_n\in A_p^{s(\cdot,\cdot)}(u_n)$, 
we derive $L^2(\Omega)\subset\mathrm{R}(\mathrm{Id}+A_p^{s(\cdot,\cdot)})$, 
thereby the operator $A_p^{s(\cdot,\cdot)}$ is $m$-completely accretive. 

\textbf{Step 3} Since the operator $A_p^{s(\cdot,\cdot)}$ is $m$-completely accretive in $L^2(\Omega)$, we let $u_{n}\in\mathrm{D}\big(A_{p}^{s(\cdot,\cdot)}\big),\;
n(v'_n-u_n)=v_n\in A_{p}^{s(\cdot,\cdot)}(u_n)$ for any $v'_n\in X_0^{s(\cdot,\cdot),p}(\Omega)\cap L^2(\Omega)$ and $\varphi=v'_n-u_n$, then we can get
\begin{equation*}
\begin{split}
    &\int_{\Omega}n\big({v'_n}(x)-{u_n}(x)\big)^2dx\\
    &+\frac12\int_{D_\Omega}|\widetilde{u}_n(x,y)|^{p-2}\widetilde{u}_n(x,y)\cdot\Big[\big({v'_n}(x)-{v'_n}(y)\big)-\widetilde{u}_n(x,y)\Big]d\sigma\\
    &=0.
\end{split}
\end{equation*}
due to inequality $\displaystyle|m|^{p-2}m(n-m)\leq\frac{1}{p}(|n|^{p}-|m|^{p})$ there holds
\begin{equation*}
\begin{split}
    &\int_{\Omega}\big({v'_n}(x)-{u_n}(x)\big)^2dx\\
    &=\frac{1}{2n}\int_{D_\Omega}|\widetilde{u}_n(x,y)|^{p-2}\widetilde{u}_n(x,y)\big({v'_n}(y)-{v'_n}(x)\big)d\sigma-\frac{1}{2n}\int_{D_\Omega}|\widetilde{u}_n(x,y)|^p d\sigma\\
    &\leq\frac{1}{2n}\int_{D_\Omega}|\widetilde{u}_n(x,y)|^{p-2}\widetilde{u}_n(x,y)\big({v'_n}(y)-{v'_n}(x)\big)d\sigma\\
    &\leq\frac{1}{2np}\int_{D_\Omega}|{v'_n}(y)-{v'_n}(x)+{u_n}(x)-{u_n}(y)|^p d\sigma-\frac{1}{2np}\int_{D_\Omega}|\widetilde{u}_n(x,y)|^p d\sigma\\
    &\leq\frac{1}{2{n^{p+1}}p}\int_{D_\Omega}|{v_n}(y)-{v_n}(x)|^p d\sigma\\
    &=\frac{1}{2{n^{p+1}}p}[v_n]_{W^{s(\cdot,\cdot),p}(\Omega)}^p.    
\end{split}
\end{equation*}

We conclude that $u_n \to v'_n$ in $L^2(\Omega)$, thus
$X_0^{s(\cdot,\cdot),p}(\Omega)\cap L^2(\Omega)\subset\overline{\mathrm{D}(A_p^{s(\cdot,\cdot)})}^{L^2(\Omega)}$. In a word, $\mathrm{D}(A_p^{s(\cdot,\cdot)})$ is dense in $L^2(\Omega)$. 

\textbf{Step 4} Let $u'_{n}\in  X_0^{s(\cdot,\cdot),p}(\Omega)\cap  L^2(\Omega),\;v_n\in A_{p}^{s(\cdot,\cdot)}(u_n)$ and $u'_n-u_n =\varphi\in X_0^{s(\cdot,\cdot),p}(\Omega)\cap L^2(\Omega)$, then we can get
\begin{equation*}
\begin{split}
    &\int_{\Omega}{v_n}\big({u'_n}(x)-{u_n}(x)\big)dx\\
    &+\frac12\int_{D_\Omega}|\widetilde{u}_n(x,y)|^{p-2}\widetilde{u}_n(x,y)\cdot\big({u'_n}(x)-{u_n}(x)-{u'_n}(y)+{u_n}(y)\big)d\sigma\\
    &=0.
\end{split}
\end{equation*}
Similar to above arguments, we have
\begin{equation*}
\begin{split}
    &\int_{\Omega}{v_n}\big({u'_n}(x)-{u_n}(x)\big)dx\\
    &=\frac{1}{2}\int_{D_\Omega}|\widetilde{u}_n(x,y)|^{p-2}\widetilde{u}_n(x,y)\big({u'_n}(y)-{u'_n}(x)-{u_n}(y)+{u_n}(x)\big)d\sigma\\
    &\leq\frac{1}{2}\int_{D_\Omega}|\widetilde{u_n}(x,y)|^{p-2}\widetilde{u_n}(x,y)\Big[\big({u'_n}(y)-{u'_n}(x)\big)-\big({u_n}(y)-{u_n}(x)\big)\Big]d\sigma\\
    &\leq\frac{1}{2p}\int_{D_\Omega}|{u'_n}(y)-{u'_n}(x)|^p d\sigma-\frac{1}{2p}\int_{D_\Omega}|{u_n}(y)-{u_n}(x)|^p d\sigma\\
    &=\mathcal{A}_{p}^{s\:(\cdot,\cdot)}(u'_n)-\mathcal{A}_{p}^{s\:(\cdot,\cdot)}(u_n).
\end{split}
\end{equation*}
Hence $v_n\in\partial\mathcal{A}_{p}^{s(\cdot,\cdot)}(u_n)$, that is to say $A_p^{s(\cdot,\cdot)}\subset\partial\mathcal{A}_{p}^{s(\cdot,\cdot)}$ as the operator $A_p^{s(\cdot,\cdot)}$ is $m$-completely accretive in $L^2(\Omega)$.
\end{proof}

\begin{theorem} \label{result6}
Let $u_0\in L^2(\Omega)$,\; there exists a unique weak solution of the problem \eqref{Eqn4} on $[0, T]$ for any $T> 0$. 
\end{theorem}

\begin{proof}
Combining Theorem \ref{result4} and Theorem \ref{result5}, there exists a unique strong solution for the abstract Cauchy problem
\begin{equation*}
\begin{cases}
    \displaystyle
    \frac{d{u_n}}{dt}+A_p^{s(\cdot,\cdot)}{u_n}=f_n,\\
    u(0)=u_0.
\end{cases}
\end{equation*}
The notion of a weak solution to the problem \eqref{Eqn4} aligns with the concept of a strong solution for above equation, thereby confirming both the existence and uniqueness of the weak solution.
\end{proof}

Next, we will prove an important conclusion by Lemma 3.6 \cite{ref6} that will be utilized in the subsequent discussion.

\begin{lemma} \label{L3}
Let $u_n\in X_0^{s(\cdot,\cdot),p}(\Omega)$ be an increasing sequence such that $u_n\geq0$, assume $T_k(u_n)\in X_0^{s(\cdot,\cdot),p}(\Omega)$ is bounded for all $k > 0$ and $T_k(u_n)$ 
 weakly converges to $T_k(u)$. Then there exists a measurable function $u$ such that $u_n\to u$ and $ T_k(u_n)\to T_k(u)$ strongly in $X_0^{s(\cdot,\cdot),p}(\Omega)$.\\
\end{lemma}

\begin{proof}
Let the sequence $u_n$ be a weak solution, taking $\varphi=T_k(u_n)-T_k(u)$ as a test function, thus
\begin{equation*}
\begin{split}
    &\int_\Omega(u_n)_t\varphi dx+\frac{1}{2}\int_{D_\Omega}|\widetilde{u}_n(x,y)|^{p-2}\widetilde{u}_n(x,y)\\
    &\cdot\Big[T_k(u_n)(x)-T_k(u)(x)-T_k(u_n)(y)+T_k(u)(y)\Big] d\sigma\\
    &=\int_{\Omega}f\big(T_k(u_n)-T_k(u)\big) dx
\end{split}
\end{equation*}
Donate
\begin{align*}
    U(x,y)&=|\widetilde{u}_n(x,y)|^{p-2}\widetilde{u}_n(x,y), \\
    T(x,y)&=\big|T_k(u_n)(x)-T_k(u_n)(y)\big|^{p-2}\cdot\Big[T_k(u_n)(x)-T_k(u_n)(y)\Big]. 
\end{align*}
The above equation can be rewritten as:
    \begin{equation*}
         A\leq B+\int_{\Omega}f_n\big(T_k(u_n)-T_k(u)\big) dx,
    \end{equation*}
where
\begin{equation*}
\begin{split}
    A&=\int_{D_\Omega}|\widetilde{u}_n(x,y)|^{p-2}\widetilde{u}_n(x,y)\cdot\Big[T_{k}(u_{n})(x)-T_{k}(u_{n})(y)\Big] d\sigma\\
    &=\int_{D_\Omega}\big|T_{k}(u_{n})(x)-T_{k}(u_{n})(y)\big|^p d\sigma\\
    &+\int_{D_\Omega}\big(U(x,y)-T(x,y)\big)\Big[T_{k}(u_{n})(x)-T_{k}(u_{n})(y)\Big] d\sigma.
\end{split}
\end{equation*}
Owing to Young's inequality
    \begin{equation*}
        T(x,y)\Big[T_{k}(u)(x)-T_{k}(u)(y)\Big]\leq\frac{1}{m}\big(T(x,y)\big)^m+\frac{1}{n}\big[T_{k}(u)(x)-T_{k}(u)(y)\big]^n
    \end{equation*}
where $\displaystyle\frac{1}{m}+\frac{1}{n}=1,\;m=\frac{p-1}{p},\;n=\frac{1}{p}$, applied in the following equation
\begin{equation*}
\begin{split}
    B&=\int_{D_\Omega}|\widetilde{u}_n(x,y)|^{p-2}\widetilde{u}_n(x,y)\cdot\Big[T_{k}(u)(x)-T_{k}(u)(y)\Big] d\sigma\\
    &=\int_{D_\Omega}T(x,y)\Big[T_{k}(u)(x)-T_{k}(u)(y)\Big] d\sigma\\
    &+\int_{D_\Omega}\big(U(x,y)-T(x,y)\big)\Big[T_{k}(u)(x)-T_{k}(u)(y)\Big] d\sigma\\
    &\leq\int_{D_\Omega}\frac{p-1}{p}\big|T_{k}(u_n)(x)-T_{k}(u_n)(y)\big|^p d\sigma\\
    &+\int_{D_\Omega}\frac{1}{p}\big|T_{k}(u)(x)-T_{k}(u)(y)\big|^p d\sigma\\
    &+\int_{D_\Omega}\big(U(x,y)-T(x,y)\big)\Big[T_k(u)(x)-T_k(u)(y)\Big] d\sigma.
\end{split}
\end{equation*}
Thus
\begin{equation*}
\begin{split}
    &\int_{D_\Omega}\frac{1}{p}\big|T_{k}(u_{n})(x)-T_{k}(u_{n})(y)\big|^p d\sigma\\
    &+\int_{D_\Omega}\big(U(x,y)-T(x,y)\big)\cdot\Big[T_{k}(u_{n})(x)-T_{k}(u)(x)-T_{k}(u_{n})(y)+T_{k}(u)(y)\Big]d\sigma\\
    &\leq\int_{D_\Omega}\frac{1}{p}\big|T_{k}(u)(x)-T_{k}(u)(y)\big|^p d\sigma
    +\int_{\Omega}f_{n}\big(T_{k}(u_{n})-T_{k}(u)\big) dx.
\end{split}
\end{equation*}

We assume $S(x,y)$ a.e. in ${D_\Omega}$ defined as
    \begin{equation*}
        S(x,y)=\big(U(x,y)-T(x,y)\big)\Big[\big(T_{k}(u_{n})(x)-T_{k}(u)(x)\big)-\big(T_{k}(u_{n})(y)-T_{k}(u)(y)\big)\Big], 
    \end{equation*}
In the second term of the left side, for $u(x)\geq u_{n}(x),\;u(y)\geq u_{n}(y)$, we divide ${D_\Omega}$ into the following four parts: ${D_\Omega}= Q_1\cup Q_2\cup Q_3\cup Q_4$ 
\begin{align*}
    Q_1&=\{(x,y)\in {D_\Omega}:u_n(x)\leq k,\;u_n(y)\leq k\},\\
    Q_2&=\{(x,y)\in {D_\Omega}:u_n(x)\geq k,\;u_n(y)\geq k\},\\
    Q_3&=\{(x,y)\in {D_\Omega}:u_n(x)\leq k,\;u_n(y)\geq k\},\\
    Q_4&=\{(x,y)\in {D_\Omega}:u_n(x)\geq k,\; u_n(y)\leq k\}.
\end{align*}
(1) If $(x,y)\in Q_{1}$, we have $U(x,y)-T(x,y)=0$, then $S(x,y)=0$.\\
(2) If $(x,y)\in Q_{2}$, the same goes for $S(x,y)=0$.\\
(3) If $(x,y)\in Q_{3}$, we have $u(x)\geq u_{n}(x),\;u(y)\geq u_{n}(y)\geq k$, then
\begin{align*}
    \big[T_{k}(u_{n})(x)-T_{k}(u)(x)\big]&-\big[T_{k}(u_{n})(y)-T_{k}(u)(y)\big]=T_{k}(u_{n})(x)-T_{k}(u)(x)\leq0, \\
    U(x,y)-T(x,y)
    &=\big(u_n(x)-u_n(y)\big)^{p-1}-\big[T_{k}(u_{n})(x)-T_{k}(u_n)(y)\big]^{p-1}\\
    &=\big(u_n(x)-u_n(y)\big)^{p-1}-\big(u_n(x)-k\big)^{p-1}\leq0,
\end{align*}
then $S(x,y)\geq0$.\\
(4) If $(x,y)\in Q_{4}$ we can prove that $S(x,y)\geq0$ in the same way.\\
Thus $S(x,y)\geq0$ a.e.in ${D_\Omega}$, i.e.
    \begin{equation*}
        \int_{D_\Omega}\big(U(x,y)-T(x,y)\big)\Big[T_k(u_n)(x)-T_k(u)(x)-T_k(u_n)(y)+T_k(u)(y)\Big]\geq0.
    \end{equation*}
Since $T_k(u_n)\to T_k(u)$ in $L^p(\Omega)$ as $n\to\infty$, we conclude that
    \begin{equation*}
        \int_\Omega f_n\big(T_k(u_n)-T_k(u)\big) dx\to0.
    \end{equation*}
Therefore
    \begin{equation*}
        \limsup_{n\to\infty}\int_{D_\Omega}\big|T_{k}(u_{n})(x)-T_{k}(u_{n})(y)\big|^{p} d\sigma
        \leq\int_{D_\Omega}\big|T_{k}(u)(x)-T_{k}(u)(y)\big|^{p} d\sigma.
    \end{equation*}
Through the Fatou's Lemma, it can be concluded that
\begin{equation*}
\begin{split}
    \frac{1}{p}\big|T_{k}(u_n)(x)-T_{k}(u_n)(y)\big|^p d\sigma
    &=\int_{D_\Omega}\liminf_{n\to\infty}\frac{1}{p}\big|T_k(u_n)(x)-T_k(u_n)(y)\big|^p d\sigma\\
    &\leq\liminf_{n\to\infty} \int_{D_\Omega}\frac{1}{p}\big|T_k(u_n)(x)-T_k(u_n)(y)\big|^p d\sigma\\    &\leq\limsup_{n\to\infty}\int_{D_\Omega}\frac{1}{p}\big|T_k(u_n)(x)-T_k(u_n)(y)\big|^p d\sigma\\
    &\leq\int_{D_\Omega}\frac{1}{p}\big|T_k(u)(x)-T_k(u)(y)\big|^p d\sigma,
\end{split}
\end{equation*}
hence
    \begin{equation*}
        \lim_{n\to\infty}\int_{D_\Omega}\frac{1}{p}\big|T_{k}(u_n)(x)-T_{k}(u_n)(y)\big|^p d\sigma
        =\int_{D_\Omega}\frac{1}{p}\big|T_k(u)(x)-T_k(u)(y)\big|^p d\sigma.
    \end{equation*}
Then by Fatou's lemma and
\begin{equation*}
\begin{split}
    &\displaystyle\big|\big(T_k(u_n)(x)-T_k(u_n)(y)\big)-\big(T_k(u)(x)-T_k(u)(y)\big)\big|^p\\
    &\leq2^p\Big[\big|T_k(u_n)(x)-T_k(u_n)(y)\big|^p+\big|T_k(u)(x)-T_k(u)(y)\big|^p\Big],
\end{split}
\end{equation*}
we have
\begin{equation*}
\begin{split}
    &\int_{D_\Omega}\frac{2^{p+1}}{p}\big|T_{k}(u)(x)-T_{k}(u)(y)\big|^p d\sigma\\
    &=\int_{D_\Omega}\frac{1}{p}\liminf_{n\to\infty}\Big(2^p\big|T_{k}(u_{n})(x)-T_{k}(u_{n})(y)\big|^p + 2^p\big|T_k(u)(x)-T_k(u)(y)\big|^p \\
    &-\big|\big(T_k(u_n)(x)-T_k(u_n)(y)\big)-\big(T_k(u)(x)-T_k(u)(y)\big)\big|^p\Big) d\sigma\\        &\leq\liminf_{n\to\infty}\int_{D_\Omega}\frac{1}{p}\Big(2^p\big|T_{k}(u_{n})(x)-T_{k}(u_{n})(y)\big|^p + 2^p\big|T_k(u)(x)-T_k(u)(y)\big|^p\\
    &-\big|\big(T_k(u_n)(x)-T_k(u_n)(y)\big)-\big(T_k(u)(x)-T_k(u)(y)\big)\big|^p\Big) d\sigma\\
    &=\int_{D_\Omega}\frac{2^{p+1}}{p}\big|T_{k}(u)(x)-T_{k}(u)(y)\big|^p d\sigma\\
    &-\limsup_{n\to\infty}\int_{D_\Omega}\frac{1}{p}\big|\big(T_k(u_n)(x)-T_k(u_n)(y)\big)-\big(T_k(u)(x)-T_k(u)(y)\big)\big|^p d\sigma,
\end{split}
\end{equation*}
Consequently
    \begin{equation*}
        \int_{D_\Omega}\frac{1}{p}\big|\big(T_k(u_n)(x)-T_k(u_n)(y)\big)-\big(T_k(u)(x)-T_k(u)(y)\big)\big|^p d\sigma\to0.
    \end{equation*}
We conclude
    \begin{equation*}
        T_k(u_n)\to T_k(u)\quad\text{strongly in}\quad X_0^{s(\cdot,\cdot),p}(\Omega).
    \end{equation*}
\end{proof}

\section{Proof of main results}

In this section, we will prove the main conclusions of this paper.\\
\begin{proof}[Proof of Theorem \ref{result1}]
(1) Existence of renormalized solutions.

In the approximate problems \eqref{Eqn4}, define $f_n\in L^\infty(\Omega_T),\;f\in L^1(\Omega_T),\;f_n=T_n(f)\geq0$, $f_n\to f$ in $L^1(\Omega_T)$, and $u_{0n}\in L^\infty(\Omega),\;u_0\in L^1(\Omega),\;u_{0n}=T_n(u_0)\geq0,\;u_{0n}\to u_0$ in $L^1(\Omega)$ i.e.
    \begin{equation}\label{Eqn5}
        \|f_n\|_{L^1(\Omega_T)}\leq\|f\|_{L^1(\Omega_T)},\quad
        \|u_{0n}\|_{L^1(\Omega)}\leq\|u_0\|_{L^1(\Omega)}.
    \end{equation}
By Theorem \ref{result6} and Comparison Principle, we can find a unique nonnegative weak solution $u_n\in L^p(0, T; X_0^{s(\cdot,\cdot),p}(\Omega))$ for the problem \eqref{Eqn4}. 

\textbf{Step 1} Let $u_m$ and $u_n$ be two weak solutions, choosing $T_1(u_n-u_m)\chi_{(0,t)}=\varphi\in L^p(0,T; X_0^{s(\cdot,\cdot),p}(\Omega))\cap L^\infty(\Omega_T)$ with $t\leq T$. The subsequent formula can be derived by subtracting
    \begin{equation*}
    \begin{split}
        &\int_{0}^{t}\int_{\Omega}(u_n-u_m)_{t}\cdot T_1(u_n-u_m)dxd\tau\\
        &+\frac{1}{2}\int_{0}^{t}\int_{D_\Omega}
        \Big(|\widetilde{u}_n(x,y,\tau)|^{p-2}
        \widetilde{u}_n(x,y,\tau)
        -|\widetilde{u}_m(x,y,\tau)|^{p-2}
        \widetilde{u}_m(x,y,\tau)\Big)\\
        &\cdot\Big[T_1(u_n-u_m)(x,\tau)-T_1(u_n-u_m)(y,\tau)\Big]
        d\sigma d\tau\\
        &=\int_0^t\int_{\Omega}(f_n-f_m)T_1(u_n-u_m)\chi_{(0,t)}dx d\tau.
    \end{split}
    \end{equation*}
According to the Mean Value Theorem and given the condition that $T'_1 \geq 0$
    \begin{equation*}
        T_1(u_n-u_m)(x,\tau)-T_1(u_n-u_m)(y,\tau)=T'_1(\xi_{nm})\big(\widetilde{u}_n(x,y,\tau)-\widetilde{u}_m(x,y,\tau)\big).
    \end{equation*}
Thus the second term in the above formula is non-negative and there exists $T_1(r)\leq1$ for any $r$, therefore
    \begin{equation*}
    \begin{split}
        \int_0^t\int_\Omega(u_n-u_m)_t\cdot T_1(u_n-u_m)d\tau
        &\leq\int_0^t\int_\Omega(f_n-f_m)T_1(u_n-u_m)\chi_{(0,t)}dxd\tau\\
        &\leq\int_0^t\int_\Omega(f_n-f_m)dxd
        \tau,
    \end{split}
    \end{equation*}
Then, it can be concluded that
    \begin{equation}\label{Eqn6}
    \begin{split}
        \int_\Omega\Theta_1(u_n-u_m)(t) dx
        &\leq\int_\Omega\Theta_1(u_{0n}-u_{0m}) dx+\int_0^t\int_\Omega(f_n-f_m)dxd\tau\\
        &\leq\|u_{0n}-u_{0m}\|_{L^1(\Omega)}+\|f_n-f_m\|_{L^1(\Omega_T)}.
    \end{split}
    \end{equation}
Let $\|u_{0n}-u_{0m}\|_{L^1(\Omega)}+\|f_n-f_m\|_{L^1(\Omega_T)}=\rho$ converge to 0  for $n,m\to+\infty$ caused by \eqref{Eqn5}. Owing to
    \begin{equation*}
    \begin{split}
        \int_\Omega\Theta_1(u_n-u_m)(t) dx
        &=\int_{\{|u_n-u_m|\leq1\}}\frac{|u_n-u_m|^2(t)}2 dx\\
        &+\int_{\{|u_n-u_m|>1\}}|u_n-u_m|(t) dx-\frac{1}{2},
    \end{split}
    \end{equation*}
thus
\begin{align*}
        \int_{\{|u_n-u_m|\leq1\}}|u_n-u_m|(t) dx&\leq\big(2\rho\cdot\mathrm{mean}(\Omega)\big)^{\frac{1}{2}}\\
        \int_{\{|u_n-u_m|>1\}}|u_n-u_m|(t)dx&\leq\rho,
\end{align*}   
we conclude
    \begin{equation*}
        \int_\Omega|u_n-u_m|(t)dx\leq\big(2\rho\cdot\mathrm{mean}(\Omega)\big)^{\frac{1}{2}}+\rho
    \end{equation*} 
converge to 0. Therefore $\{u_{n}\}$ is a Cauchy sequence in $C([0,T];L^1(\Omega))$ and $u_n\to u$ in $C([0,T];L^1(\Omega))$, then we find an a.e. subsequence $u_{n}\to u$ in $\Omega_T$. 

\textbf{Step 2} Recalling the following elementary algebraic inequality
    \begin{equation*}
        |T_k(n)-T_k(m)|^p\leq|n-m|^{p-2}(n-m)\big(T_k(n)-T_k(m)\big)
    \end{equation*}
for $n,m\in\mathbb{R},\;p>1$, combining \eqref{Eqn5}\eqref{Eqn6} and the definition of $T_k,\;\Theta_k$ we get
    \begin{equation}\label{Eqn7}
    \begin{split}
        &\frac{1}{2}\int_{0}^{T}\int_{D_\Omega}\big|T_{k}(u_{n})(x,t)-T_{k}(u_{n})(y,t)\big|^{p} d\sigma dt\\
        &\leq\frac12\int_0^T\int_{D_\Omega}|\widetilde{u}_n(x,y,t)|^{p-2}\widetilde{u}_n(x,y,t)\cdot\big(T_k(u_n)(x,t)-T_k(u_n)(y,t)\big) d\sigma dt\\
        &\leq k\big(\|f_n\|_{L^1(\Omega_T)}+\|u_{0n}\|_{L^1(\Omega)}\big)\\
        &\leq k\big(\|f\|_{L^1(\Omega_T)}+\|u_0\|_{L^1(\Omega)}\big).
    \end{split}
    \end{equation}
Then, we deduce that
    \begin{equation*}
        T_k(u_n)\rightharpoonup T_k(u)\quad\text{weakly in}\quad L^p(0,T;X_0^{s(\cdot,\cdot),p}(\Omega))
    \end{equation*}
    
To handle the time derivative of truncations, we will employ the regularization method \cite{ref31}. For $\epsilon > 0$, we define $\big(T_k(u)\big)_\epsilon\in L^p(0, T;X_0^{s(\cdot,\cdot),p}(\Omega))\cap L^\infty(\Omega_T)$ the time regularization of the function $T_k(u)$ given by
    \begin{equation*}       \Big|\big(T_k(u)\big)_\epsilon(x,t)\Big|:=\Big|\epsilon\int_{-\infty}^t e^{\epsilon(v-t)}T_k\big(u(x,v)\big) dv\Big|\leq k(1-e^{-\epsilon t})<k,
    \end{equation*}
a.e. in $\Omega_T$, and $T_k(u)=0$ for $v < 0$. Observe that, 
    \begin{equation*}
        \frac{\partial(T_k(u))_\epsilon}{\partial t}=\epsilon\Big(T_k(u)-\big(T_k(u)\big)_\epsilon\Big)
    \end{equation*}
is differentiable for a.e. $t\in(0,T)$, thus the sequence $\big(T_k(u)\big)_\epsilon$ as an approximation for function $T_k(u)$. Let us take a sequence $\phi_{\iota}\in C_0^\infty(\Omega)$ satisfied $\lim\limits_{{\iota}\to\infty}\left\|\phi_{\iota}-u_{0}\right\|_{L^{1}(\Omega)}=0$. To handle non-zero initial conditions, define $\varrho_{\epsilon,\iota}$ \cite{ref3} as a smooth approximation of $T_k(u)$ with $\varrho_{\epsilon,\iota}(u)=\big(T_k(u)\big)_\epsilon+e^{-\epsilon t}T_k(\phi_{\iota})$, the function has the following properties:
    \begin{equation*}
    \left\{\begin{array}{l}\big(\varrho_{\epsilon, \iota}(u)\big)_{t} = \epsilon\big(T_{k}(u)-\varrho_{\epsilon, \iota}(u)\big), \\\varrho_{\epsilon, \iota}(u)(0) = T_{k}\left(\phi_{\iota}\right), \\\left|\varrho_{\epsilon, \iota}(u)\right| \leq k, \\\varrho_{\epsilon, \iota}(u) \rightarrow T_{k}(u)\quad\text{strongly in}\quad L^{p}(0, T ; X_{0}^{s(\cdot,\cdot), p}(\Omega)),\;\text{as}\;\epsilon\to+\infty .\end{array}\right.
    \end{equation*}
We choose $\varphi_n=T_{2k}\big(u_n-T_h(u_n)+T_k(u_n)-\varrho_{\epsilon,\iota}(u)\big)$ as a test function where $h>k>0$, 
    \begin{equation*}
    \begin{split}
        &\int_{0}^{T}\int_{\Omega}(u_n)_t \varphi_n dx dt+\int_{0}^{T}\int_{D_\Omega}|\widetilde{u}_n(x,y,t)|^{p-2}\widetilde{u}_n(x,y,t)\cdot\big(\varphi_n(x,t)-\varphi_n(y,t)\big) d\sigma dt\\
        &=\int_{0}^{T}\int_{\Omega} f_n \varphi_n dx dt.
    \end{split}
    \end{equation*}
Denote $\displaystyle\lim\limits_{h\to+\infty}\lim\limits_{\iota\to+\infty}\lim\limits_{\epsilon\to+\infty}\lim\limits_{n\to+\infty}\varphi(n,\epsilon,\iota,h)=0$. Following the arguments in Theorem 1.1 \cite{ref106} we conclude that
    \begin{equation*}
    \begin{split}
        &\int_{0}^{T}\int_{D_\Omega}|\widetilde{u}_n(x,y,t)|^{p-2}\widetilde{u}_n(x,y,t)\cdot\big(T_{k}(u_{n})(x,t)-T_{k}(u_{n})(y,t)\big) d\sigma dt\\
        &\leq\int_0^T\int_{D_\Omega}|\widetilde{u}_n(x,y,t)|^{p-2}\widetilde{u}_n(x,y,t)\cdot\big(T_k(u)(x,t)-T_k(u)(y,t)\big) d\sigma dt \\
        &+\varphi(n,\epsilon,\iota,h).
    \end{split}
    \end{equation*}
According to Lemma \ref{L3}, we get
    \begin{equation*}
        T_k(u_n)\to T_k(u)\quad\text{strongly in}\quad L^p(0,T; X_0^{s(\cdot,\cdot),p}(\Omega)).
    \end{equation*}
    
\textbf{Step 3} Choosing $\varphi=T_1\big(u_n-T_k(u_n)\big)$ as a test function for Definition \ref{D2} and given $h>0$, we find
    \begin{equation*}
    \begin{split}
        &\int_{\{|u_{n}|>h\}}\Theta_{1}(u_{n}\mp h)(T) dx-\int_{\{|u_{0n}|>h\}}\Theta_{1}(u_{0n}\mp h) dx \\
        &+\frac12\int_0^T\int_{D_\Omega}|\widetilde{u}_n(x,y,t)|^{p-2}\widetilde{u}_n(x,y,t)\\
        &\cdot\Big[T_1\big(u_n-T_k(u_n)\big)(x,t)-T_1\big(u_n-T_k(u_n)\big)(y,t)\Big] d\sigma dt \\
        &\leq\int_\Omega f_nT_1\big(u_n-T_k(u_n)\big) dxdt.
    \end{split}
    \end{equation*}
From \eqref{Eqn7} it can be known that
    \begin{equation*}
    \begin{split}
        &\frac{1}{2}\int_{0}^{T}\int_{D_\Omega}\big|T_{1}\big(u_n-T_k(u_n)\big)(x,t)-T_{1}\big(u_n-T_k(u_n)\big)(y,t)\big|^{p} d\sigma dt\\
        &\leq\frac{1}{2}\int_0^T\int_{D_\Omega}|\widetilde{u}_n(x,y,t)|^{p-2}\widetilde{u}_n(x,y,t)\\
        &\cdot\big[T_1\big(u_n-T_k(u_n)\big)(x,t)-T_1\big(u_n-T_k(u_n)\big)(y,t)\big] d\sigma dt\\
       &\leq\|f_n\|_{L^1(\Omega_T)}+\|u_{0n}\|_{L^1(\Omega)}\\
       &\leq\int_{\{|u_n|>h\}}|f_n| dxdt+\int_{\{|u_{0n}|>h\}}|u_{0n}| dx. 
    \end{split}
    \end{equation*}
It is easy to obtain
    \begin{equation*}
    \begin{split}
        &|\widetilde{u}_n(x,y,t)|^{p-2}\widetilde{u}_n(x,y,t)\Big[T_1\big(u_n-T_k(u_n)\big)(x,t)-T_1\big(u_n-T_k(u_n)\big)(y,t)\Big]\\
        &=T'_1\big(\xi_n-T_k(\xi_n)\big)\big(1-T'_k(\xi_n)\big)|\widetilde{u}_n(x,y,t)|^p\geq0,\;\mathrm{as}\;T'_1>0
    \end{split}
    \end{equation*}
Recalling ${u_n\to u}$ in $C([0,T];L^1(\Omega))$, we have $\lim\limits_{h\to+\infty}\mathrm{meas}\{(x,t)\in\Omega_T:|u_n|>h\}=0$ for all $n$.
Therefore for all $\big(u_n(x,t),u_n(y,t)\big) \in D_h$, we have
\begin{equation*}
\begin{split}
    |\widetilde{u}_n(x,y,t)|^{p-2}\widetilde{u}_n(x,y,t)\Big[T_1\big(u_n-T_k(u_n)\big)(x,t)-T_1\big(u_n-T_k(u_n)\big)(y,t)\Big]\\
    \geq|\widetilde{u}_n(x,y,t)|^{p-1}.
\end{split}
\end{equation*}
By using Fatou’s lemma we obtain the renormalized condition as $n\to\infty$
    \begin{equation}\label{Eqn8}
        \lim_{h\to+\infty}\iiint_{\{(u(x,t),u(y,t))\in D_h\}}|\widetilde{u}(x,y,t)|^{p-1} d\sigma dt=0.
    \end{equation}
    
\textbf{Step4} Let $H \in W^{1,\infty}(\mathbb{R})$ and $\mathrm{supp} H'\subset[-h,h]$ with $h > 0$. For every $\phi \in C^1(\bar{\Omega}_T)$ with $\phi=0$ in $\Omega^c\times(0,T)$ and $\phi(\cdot,T)=0$ in $\Omega$, choosing $\varphi=H'(u_n)\phi$ in Definition \ref{D1}
    \begin{equation*}
    \begin{split}
        &\int_0^T\int_\Omega\big(H(u_n)\big)_t\phi dx dt\\    &+\frac12\int_0^T\int_{D_\Omega}|\widetilde{u}_n(x,y,t)|^{p-2}\widetilde{u}_n(x,y,t)\big(H'(u_n)(x,t)-H'(u_n)(y,t)\big)\\
        &\cdot\frac{\phi(x,t)+\phi(y,t)}2 d\sigma dt\\        &+\frac12\int_0^T\int_{D_\Omega}|\widetilde{u}_n(x,y,t)|^{p-2}\widetilde{u}_n(x,y,t)\frac{H'(u_n)(x,t)+H'(u_n)(y,t)}2\\
        &\cdot\big(\phi(x,t)-\phi(y,t)\big)d\sigma dt \\
        &=\int_0^T\int_\Omega f_n\big(H'(u_n)\phi\big)dxdt,
    \end{split}
    \end{equation*}
the above formula can be rewritten as
    \begin{equation*}
        \int_0^T\int_\Omega \big(H(u_n)\big)_t\phi dx dt+A+B
        =\int_0^T\int_\Omega f_n\big(H'(u_n)\phi\big)dxdt.
    \end{equation*}
Since $H$ is bounded and continuous, $u_n\to u$ a.e. in $\Omega_T$ leads to 
$H(u_n)\to H(u)$ a.e. in $\Omega_T$ and 
$H(u_n)\rightharpoonup H(u)$ weakly* in $L^\infty(\Omega_T)$, 
then $\big(H(u_n)\big)_t\to \big(H(u)\big)_t$ in $D'(\Omega_T)$ as $n\to+\infty$, that is
    \begin{equation*}
        \int_0^T\int_\Omega\big(H(u_n)\big)_t\phi dxdt
        \to\int_0^T\int_\Omega\big(H(u)\big)_t\phi dxdt
    \end{equation*}
and
    \begin{equation*}
        \int_0^T\int_\Omega\big(H(u)\big)_t\phi dx dt
        =-\int_{0}^{T}\int_\Omega H(u)\frac{\partial\phi}{\partial t}dx dt
        -\int_\Omega H(u_0)\phi(x_0)dx dt
    \end{equation*}
For the right side of the equation, since $f_n\to f$ in $L^1(\Omega_T)$, as $n\to+\infty$, there is
    \begin{equation*}
        \int_0^T\int_\Omega f_n\big(H'(u_n)\phi\big)dxdt\to\int_0^T\int_\Omega f\big(H'(u)\phi\big)dxdt
    \end{equation*}
For the second term A, assume that $\mathrm{supp} H'\subset[-h,h]$, let ${D_\Omega}\times(0,T) = Q_1\cup Q_2\cup Q_3\cup Q_4\cup Q_5\cup Q_6$
\begin{align*}
    Q_1&=\{(x,y,t)\in {D_\Omega}\times(0,T):u_n(x,t)\geq h,\;u_n(y,t)\geq h\},\\
    Q_2&=\{(x,y,t)\in {D_\Omega}\times(0,T):u_n(x,t)\leq h,\;u_n(y,t)\leq h\},\\
    Q_3&=\{(x,y,t)\in {D_\Omega}\times(0,T):h\leq u_n(x,t)\leq h+1,\;u_n(y,t)\leq h\},\\
    Q_4&=\{(x,y,t)\in {D_\Omega}\times(0,T):u_n(x,t)\geq h+1,\;u_n(y,t)\leq h\},\\
    Q_5&=\{(x,y,t)\in {D_\Omega}\times(0,T):u_n(x,t)\leq h,\;h\leq u_n(y,t)\leq h+1\},\\
    Q_6&=\{(x,y,t)\in {D_\Omega}\times(0,T):u_n(x,t)\leq h,\;u_n(y,t)\geq h+1\}
\end{align*}
(1) If $(x,y,t)\in Q_{1}$ we have $H'(u_n)(x,t)=H’(u_n)(y,t)=0$, then $Q_1=0$.\\
(2) If $(x,y,t)\in Q_{2}$ we have  $u_n(x,t)=T_k(u_n)(x,t),\;u_n(y,t)=T_k(u_n)(y,t)$. Due to $T_k(u_n)\to T_k(u)$ strongly in $L^p(0,T; X_0^{s(\cdot,\cdot),p}(\Omega))$, we know that
    \begin{equation*}
    \begin{split}
        \frac{\big|T_{k}(u_{n})(x,t)-T_{k}(u_{n})(y,t)\big|^{p-2}\big(T_{k}(u_{n})(x,t)-T_{k}(u_{n})(y,t)\big)}{|x-y|^{\frac{(N+p{s\left(x,y\right)})(p-1)}p}} \\\to\frac{\big|T_k(u)(x,t)-T_k(u)(y,t)\big|^{p-2}\big(T_k(u)(x,t)-T_k(u)(y,t)\big)}{|x-y|^{\frac{(N+p{s\left(x,y\right)})(p-1)}{p}}}
    \end{split}
    \end{equation*}
strongly in $\displaystyle L^{\frac p{p-1}}({D_\Omega}\times(0,T))$.\\
Moreover, $u_n\to u$ a.e. in $\Omega_T, H\in W^{1,\infty}(\mathbb{R})$ and $\phi\in C^1(\overline{\Omega}_T)$ with $\phi=0$ in ${D_\Omega}\times(0,T)$ imply that
    \begin{equation*}
        H\big(T_k(u_n)(x,t)\big)-H\big(T_k(u_n)(y,t)\big)=H'(\xi_n)\big[T_k(u_n)(x,t)-T_k(u_n)(y,t)\big].
    \end{equation*}
Obviously
    \begin{equation*}
        \left\{\frac{H\big(T_k(u_n)(x,t)\big)-H\big(T_k(u_n)(y,t)\big)}{|x-y|^{\frac{N+ps(x,y)}{p}}}\cdot\frac{\phi(x,t)+\phi(y,t)}{2}\cdot\chi_{Q_2}\right\}_n
    \end{equation*}
is bounded in $L^p({D_\Omega}\times(0,T))$, 
    \begin{equation*}
    \begin{split}
        &\frac{H^{\prime}(u_{n})(x,t)-H^{\prime}(u_{n})(y,t)}{|x-y|^{\frac{N+p{s\left(x,y\right)}}{p}}}\cdot\frac{\phi(x,t)+\phi(y,t)}{2}\chi_{Q_2} \\
        &\rightharpoonup\frac{H'(u)(x,t)-H'(u)(y,t)}{|x-y|^{\frac{N+p{s\left(x,y\right)}}{p}}}\cdot\frac{\phi(x,t)+\phi(y,t)}{2}\chi_{Q_2}
    \end{split}
    \end{equation*}
weakly in $L^{p}(D_\Omega\times(0,T))$. Thus, as $n\to\infty$
    \begin{equation*}
    \begin{split}
        &\iint_{Q_2}|\widetilde{u}_n(x,y,t)|^{p-2}\widetilde{u}_n(x,y,t)\big(H'(u_n)(x,t)-H'(u_n)(y,t)\big)\cdot\frac{\phi(x,t)+\phi(y,t)}2 d\sigma dt\\
        &\to\iint_{Q_{2}}|\widetilde{u}(x,y,t)|^{p-2}\widetilde{u}(x,y,t)\big(H'(u)(x,t)-H'(u)(y,t)\big)\cdot\frac{\phi(x,t)+\phi(y,t)}2 d\sigma dt.
    \end{split}
    \end{equation*}
(3) If $(x,y,t)\in Q_{3}$, it is similar to $Q_{2}$, the same applies to $Q_{5}$.\\
(4) If $(x,y,t)\in Q_{4}$, we derive
    \begin{equation*}
        \lim_{h\to\infty}\lim_{n\to\infty}\iiint_{\{(u_n(x,t),u_n(y,t))\in D_h\}}|\widetilde{u}_n(x,y,t)|^{p-1} d\sigma dt=0,
    \end{equation*}
based on \eqref{Eqn8}. Then
    \begin{equation*}
    \begin{split}
    &\lim_{h\to\infty}\lim_{n\to\infty}\iint_{Q_4}|\widetilde{u}_n(x,y,t)|^{p-2}\widetilde{u}_n(x,y,t)\\
    &\cdot\big(H'(u_n)(x,t)-H'(u_n)(y,t)\big)\frac{\phi(x,t)+\phi(y,t)}2 d\sigma dt=0.
    \end{split}
    \end{equation*}
It also follows similarly for $Q_{6}$.
Therefore, we have
    \begin{equation*}
    \begin{split}
        &\lim_{h\to\infty}\lim_{n\to\infty}A \\&=\lim_{h\to\infty}\iint_{Q_{2}}|\widetilde{u}(x,y,t)|^{p-2}\widetilde{u}(x,y,t)\big(H^{\prime}(u)(x,t)-H^{\prime}(u)(y,t)\big)\cdot\frac{\phi(x,t)+\phi(y,t)}2 d\sigma dt \\
        &+2\lim_{h\to\infty}\iint_{Q_{3}}|\widetilde{u}(x,y,t)|^{p-2}\widetilde{u}(x,y,t)\big(H^{\prime}(u)(x,t)-H^{\prime}(u)(y,t)\big)\cdot\frac{\phi(x,t)+\phi(y,t)}2 d\sigma dt \\
        &=\int_0^T\int_{D_\Omega}|\widetilde{u}(x,y,t)|^{p-2}\widetilde{u}(x,y,t)\big(H^{\prime}(u)(x,t)-H^{\prime}(u)(y,t)\big)\cdot\frac{\phi(x,t)+\phi(y,t)}2 d\sigma dt.
    \end{split}
    \end{equation*}
Similar to the same arguments as before for the third term B. 

Therefore, we obtain
\begin{equation*}
\begin{split}
    &\int_0^T\int_\Omega \big(H(u)\big)_t\phi dx dt\\   &+\frac12\int_0^T\int_{D_\Omega}|\widetilde{u}(x,y,t)|^{p-2}\widetilde{u}(x,y,t)\big(H'(u)(x,t)-H'(u)(y,t)\big)\cdot\frac{\phi(x,t)+\phi(y,t)}2 d\sigma dt\\        &+\frac12\int_0^T\int_{D_\Omega}|\widetilde{u}(x,y,t)|^{p-2}\widetilde{u}(x,y,t)\frac{H'(u)(x,t)+H'(u)(y,t)}2\cdot\big(\phi(x,t)-\phi(y,t)\big) d\sigma dt \\    
    &=\int_0^T\int_\Omega f\big(H'(u)\phi\big) dx dt.
\end{split}
\end{equation*}
that is
\begin{equation*}
\begin{split}
    &-\int_{0}^{T} \int_{\Omega}H(u)\frac{\partial\phi}{\partial t} dxdt 
    -\int_{\Omega}H(u_{0})\phi(x,0) dx\\
    &+\frac{1}{2}\int_{0}^{T}\int_{D_\Omega}|\widetilde{u}(x,y,t)|^{p-2}\widetilde{u}(x,y,t)\big[\big(H'(u)\phi\big)(x,t)-\big(H'(u)\phi\big)(y,t)\big] d\sigma dt \\
    &=\int_0^T\int_\Omega f\big(H'(u)\phi\big) dxdt
\end{split}
\end{equation*}
for any $\phi\in C^1(\overline{\Omega}_T)$ with $\phi=0$ in $\Omega^c\times(0,T)$ and $\phi(\cdot,T)=0$ in $\Omega$. 

(2) Uniqueness of renormalized solutions.

For $h > 0$, let function $H_h\in W^{1,\infty}(\mathbb{R})$ with $\mathrm{supp}H'_h\subset[-h-1,h+1]$ satisfied
\begin{equation*}
    \left.\left\{\begin{array}{ll}H_h(r)=r&\text{if}\;|r|<h,\\H_h(r)=\left(h+\frac{1}{2}\right)\mp\frac{1}{2}(r\mp(h+1))^2&\text{if}\;h\leq|r|\leq h+1,\\H_h(r)=\pm\left(h+\frac{1}{2}\right)&\text{if}\;|r|>h+1.\end{array}\right.\right.
\end{equation*}
and
\begin{equation*}
    \left\{\begin{array}{ll}H_h'(r)=1&\text{if}\;|r|<h,\\H_h'(r)=h+1-|r|&\text{if}\;h\leq|r|\leq h +1,\\H_h'(r)=0&\text{if}\;|r|>h + 1.\end{array}\right.
\end{equation*}
Let $u$ and $v$ be two renormalized solutions for problem \eqref{Quotient}, subtract one equation from the other, and for every fixed $k > 0$, let $\varphi=H'_h(u)T_{k}\big(H_{h}(u)\big)$ with $\phi=T_k\big(H_h(u)\big)$ as a test function, therefore we have
     \begin{equation*}
     \begin{split}
         &\int_0^T\int_{\Omega}\big(H_h(u)-H_h(v)\big)_t\cdot T_k\big(H_h(u)-H_h(v)\big)dxdt\\
         &+\frac{1}{2}\int_{0}^{T}\int_{D_\Omega} \Big[|\widetilde{u}(x,y,t)|^{p-2}\widetilde{u}(x,y,t)\cdot\big(H'_{h}(u)(x,t)-H'_{h}(u)(y,t)\big)\\
         &-|\widetilde{v}(x,y,t)|^{p-2}\widetilde{v}(x,y,t)\cdot\big(H’_{h}(v)(x,t)-H'_{h}(v)(y,t)\big)\Big] \\
         &\cdot\frac{T_{k}\big(H_{h}(u)-H_{h}(v)\big)(x,t)+T_{k}\big(H_{h}(u)-H_{h}(v)\big)(y,t)}{2} d\sigma dt\\
         &+\frac{1}{2}\int_{0}^{T}\int_{D_\Omega}\Big(|\widetilde{u}(x,y,t)|^{p-2}\widetilde{u}(x,y,t)\frac{H'_{h}(u)(x,t)+H'_{h}(u)(y,t)}{2}\\
         &-|\widetilde{v}(x,y,t)|^{p-2}\widetilde{v}(x,y,t)\frac{H'_{h}(v)(x,t)+H'_{h}(v)(y,t)}{2}\Big)\\
         &\cdot\Big[T_{k}\big(H_{h}(u)-H_{h}(v)\big)(x,t)-T_{k}\big(H_{h}(u)-H_{h}(v)\big)(y,t)\Big] d\sigma dt\\
         &=\int_0^T\int_\Omega f\big(H'_h(u)-H'_h(v)\big)T_k\big(H_h(u)-H_h(v)\big) dxdt
     \end{split}
     \end{equation*}
and express the above equation as $A+B+C=Z$. Setting $|u|,|v|\leq k\leq h$, we have $T_k(r)=r, H_h(r)=r, H'_h(r)=1$, thus $T_k\big(H_h(u)-H_h(v)\big)=T_k(u-v)=u-v$,
\begin{equation*}
\begin{split}    A&=\int_{\Omega}\Theta_{k}\big(H_{h}(u)-H_{h}(v)\big)(T) dx-\int_{\Omega}\Theta_{k}\big(H_{h}(u)-H_{h}(v)\big)(0) dx\\         &=\int_\Omega\Theta_k\big(H_h(u)-H_h(v)\big)(T) dx\geq0.
\end{split}
\end{equation*}
Assume $\displaystyle\frac{u(x,t)+u(y,t)-v(x,t)-v(y,t)}{2}\leq M$, then
     \begin{equation*}
     \begin{split}
         |B|&\leq\frac{1}{2}\int_{0}^{T}\int_{D_\Omega} \Big[|\widetilde{u}(x,y,t)|^{p-2}\widetilde{u}(x,y,t)\cdot\big(H'_{h}(u)(x,t)-H'_{h}(u)(y,t)\big)\\
         &+|\widetilde{v}(x,y,t)|^{p-2}\widetilde{v}(x,y,t)\cdot\big(H'_{h}(v)(x,t)-H'_{h}(v)(y,t)\big)\Big] \\
         &\cdot\frac{T_{k}\big(H_{h}(u)-H_{h}(v)\big)(x,t)+T_{k}\big(H_{h}(u)-H_{h}(v)\big)(y,t)}{2} d\sigma dt\\
         &\leq M\Big(\iiint_{\{(u(x,t),u(y,t))\in D_h\}}|\widetilde{u}(x,y,t)|^{p-1} d\sigma dt\\
         &+\iiint_{\{(v(x,t),v(y,t))\in D_h\}}|\widetilde{v}(x,y,t)|^{p-1} d\sigma dt\Big),
     \end{split}
     \end{equation*}
we obtain $B\to 0$ as $h\to\infty$ by \eqref{Eqn8}. We can divide the third term into three parts, by denoting
\begin{equation*}
\begin{split}
    C&=\frac{1}{2}\int_{0}^{T} \int_{D_\Omega}\big(|\widetilde{u}(x,y,t)|^{p-2}\widetilde{u}(x,y,t)-|\widetilde{v}(x,y,t)|^{p-2}\widetilde{v}(x,y,t)\big) \\
    &\cdot\Big[T_k\big(H_h(u)-H_h(v)\big)(x,t)-T_k\big(H_h(u)-H_h(v)\big)(y,t)\Big] d\sigma dt \\
    &+\frac{1}{2}\int_{0}^{T}\int_{D_\Omega}|\widetilde{u}(x,y,t)|^{p-2}\widetilde{u}(x,y,t)\Big(\frac{H_{h}^{\prime}(u)(x,t)+H_{h}^{\prime}(u)(y,t)}{2}-1\Big) \\
    &\cdot\Big[T_k\big(H_h(u)-H_h(v)\big)(x,t)-T_k\big(H_h(u)-H_h(v)\big)(y,t)\Big] d\sigma dt \\
    &+\frac{1}{2}\int_{0}^{T}\int_{D_\Omega}|\widetilde{v}(x,y,t)|^{p-2}\widetilde{v}(x,y,t)\Big(1-\frac{H'_{h}(v)(x,t)+H'_{h}(v)(y,t)}{2}\Big) \\
    &\cdot\Big[T_k\big(H_h(u)-H_h(v)\big)(x,t)-T_k\big(H_h(u)-H_h(v)\big)(y,t)\Big] d\sigma dt \\
    &=C_1+C_2+C_3,
\end{split}
\end{equation*}
\begin{equation*}
\begin{split}
    C_1&\geq\frac{1}{2}\iiint_{\{|u-v|\leq k\}\cap\{|u|,|v|\leq k\}}\big(|\widetilde{u}(x,y,t)|^{p-2}\widetilde{u}(x,y,t)-|\widetilde{v}(x,y,t)|^{p-2}\widetilde{v}(x,y,t)\big) \\
    &\cdot\Big[T_k\big(H_h(u)-H_h(v)\big)(x,t)-T_k\big(H_h(u)-H_h(v)\big)(y,t)\Big] d\sigma dt \\
    &=\frac{1}{2}\iiint_{\{|u-v|\leq k\}\cap\{|u|,|v|\leq k\}}\big(|\widetilde{u}(x,y,t)|^{p-2}\widetilde{u}(x,y,t)-|\widetilde{v}(x,y,t)|^{p-2}\widetilde{v}(x,y,t)\big)\\
    &\cdot\big(\widetilde{u}(x,y,t)-\widetilde{v}(x,y,t)\big) d\sigma dt.
\end{split}
\end{equation*}     
By the Lebesgue dominated convergence theorem, we conclude that $C_2, C_3\to 0$ as $h\to\infty$ and we deduce that $Z\to 0$ as $f\big(H_h'(u)-H_h'(v)\big)\to0$ strongly in $L^1(\Omega_T)$. Therefore, sending $h\to\infty$, we have
     \begin{equation*}
     \begin{split}
         &\iiint_{\{|u|\leq\frac k2,|v|\leq\frac k2\}}\big(|\widetilde{u}(x,y,t)|^{p-2}\widetilde{u}(x,y,t)-|\widetilde{v}(x,y,t)|^{p-2}\widetilde{v}(x,y,t)\big)\\
         &\cdot[\widetilde{u}(x,y,t)-\widetilde{v}(x,y,t)] d\sigma dt=0,
     \end{split}
     \end{equation*}
on the set $\displaystyle{\{|u|\leq\frac k2,|v|\leq\frac k2\}}$. It is claimed that $\widetilde{u} = \widetilde{v}$, since $k$ is arbitrary, we have $u=v$ for a.e. $x,y\in\mathbb{R}^N,\;t\in[0,T]$. 

In the Poincaré inequality \eqref{Eqn2} letting $p = 1$
     \begin{equation*}
         \int_{0}^{T} \int_\Omega|u(x,t)-v(x,t)|dxdt
         \leq C\int_{0}^{T} \iint_{D_\Omega}\frac{|\widetilde{u}(x,y,t)-\widetilde{v}(x,y,t)|}{|x-y|^{N+{s\left(x,y\right)}}} dxdydt
    \end{equation*}
yields. 
Thus we have $u = v$ a.e. in $\Omega_T$. 
\end{proof}
Subsequently, we prove that the problem admits unique entropy solution and it is equivalent to that of the renormalized solution $u$.\\
\begin{proof}[Proof of Theorem \ref{result2}]

(1) Existence of entropy solutions.

Based on the arguments \cite{ref5}, we can establish the existence of entropy solutions, while \cite{ref20} allows us to conclude the equivalence between renormalized solutions and entropy solutions. We will omit the detailed proof here.

(2) Uniqueness of entropy solutions. 

Suppose that $u$ and $v$ are two entropy solutions of problem \eqref{Quotient} and construct a sequence $u_n$. 
Choosing $\varphi_1= T_k\big(v-H_h(u_n)\big)$ with $\phi_1=H_h(u_n)$ as a test function for entropy solution $v$, we have
     \begin{equation*}
     \begin{split}
         &\int_\Omega\Theta_k\big(v-H_h(u_n)\big)(T) dx-\int_\Omega\Theta_k\big(u_0-H_h(u_{0n})\big)dx\\
         &+\int_0^T\int_\Omega(u_n)_t H'_h(u_n) T_k\big(v-H_h(u_n)\big)dxdt\\
         &+\frac12\int_0^T\int_{D_\Omega}|\widetilde{v}(x,y,t)|^{p-2}\widetilde{v}(x,y,t)\cdot\big(\varphi_1(x,t)-\varphi_1(y,t)\big)d\sigma dt\\
         &\leq\int_0^T\int_\Omega f T_k\big(v-H_h(u_n)\big) dxdt.
     \end{split}
     \end{equation*}
To eliminate the third term on the left side of the equation, we assume $u_n$ is the renormalized (entropy) solution of problem \eqref{Eqn4}, taking $\varphi_2=H'_h(u_n) T_k\big(v-H_h(u_n)\big)$ as a test function with $\phi_2=T_k\big(v-H_h(u_n)\big)$, we obtain
     \begin{equation*}
     \begin{split}
         &\int_0^T\int_\Omega(u_n)_t H'_h(u_n) T_k\big(v-H_h(u_n)\big)dxdt \\
         &+\frac12\int_0^T\int_{D_\Omega}|\widetilde{u}_n(x,y,t)|^{p-2}\widetilde{u}_n(x,y,t)\big(H'_h(u_n)(x,t)-H'_h(u_n)(y,t)\big)\\
         &\cdot\frac{\phi_2(x,t)+\phi_2(y,t)}2 d\sigma dt \\
         &+\frac12\int_0^T\int_{D_\Omega}|\widetilde{u}_n(x,y,t)|^{p-2}\widetilde{u}_n(x,y,t)\frac{H'_h(u)(x,t)+H'_h(u)(y,t)}2\\
         &\cdot\big(\phi_2(x,t)-\phi_2(y,t)\big)d\sigma dt \\
         &=\int_0^T\int_\Omega f_n H'_h(u_n)T_k\big(v-H_h(u_n)\big)dxdt.  
     \end{split}
     \end{equation*}
Taking the difference between the above two equations, then
     \begin{equation*}
     \begin{split}
         &\int_\Omega\Theta_k(v-H_h(u_n))(T) dx-\int_\Omega\Theta_k(u_0-H_h(u_{0n})) dx \\
         &-\frac12\int_0^T\int_{D_\Omega}|\widetilde{u}_n(x,y,t)|^{p-2}\widetilde{u}_n(x,y,t)(H'_h(u_n)(x,t)-H'_h(u_n)(y,t))\\
         &\cdot\frac{\varphi_1(x,t)+\varphi_1(y,t)}2 d\sigma dt \\
         &-\frac12\int_0^T\int_{D_\Omega}|\widetilde{u}_n(x,y,t)|^{p-2}\widetilde{u}_n(x,y,t)+\frac{H'_h(u)(x,t)+H'_h(u)(y,t)}2\\
         &\cdot\big(\varphi_1(x,t)-\varphi_1(y,t)\big)d\sigma dt \\
         &+\frac12\int_0^T\int_{D_\Omega}|\widetilde{v}(x,y,t)|^{p-2}\widetilde{v}(x,y,t)\cdot\big(\varphi_1(x,t)-\varphi_1(y,t)\big)d\sigma dt \\
         &\leq\int_0^T\int_\Omega\big(f-f_n H'_h(u_n)\big)T_k\big(v-H_h(u_n)\big) dxdt. 
     \end{split}
     \end{equation*}
The third term on the left is denoted as A, and the fourth and fifth terms are denoted as B. 
Let $|u|,\;|v|\leq k\leq h$, we have $T_k(r)=r,\;H_h(r)=r,\;H'_h(r)=1$, thus $\varphi_1= T_k\big(v-H_h(u_n)\big)=T_k(v-u_n)=v-u_n$, consequently $\varphi_1(x,t)+\varphi_1(y,t)\leq\widetilde{v}(x,y,t)$ and $\varphi_1(x,t)-\varphi_1(y,t)=\widetilde{v}(x,y,t)-\widetilde{u}(x,y,t)$.
     \begin{equation*}
     \begin{split}
         |A|&\leq\frac12\int_0^T\int_{D_\Omega}|\widetilde{u}_n(x,y,t)|^{p-2}\widetilde{u}_n(x,y,t)\big(H'_h(u_n)(x,t)+H'_h(u_n)(y,t)\big)\\
         &\cdot\frac{\varphi_1(x,t)+\varphi_1(y,t)}2 d\sigma dt \\
         &\leq\int_0^T\int_{D_\Omega}|\widetilde{u}_n(x,y,t)|^{p-2}\widetilde{u}_n(x,y,t)\widetilde{v}(x,y,t) d\sigma dt\\
         &\leq k\iiint_{\{(u_n(x,t),u_n(y,t))\in D_h}|\widetilde{u}_n(x,y,t)|^{p-1}d\sigma dt,
     \end{split}
     \end{equation*}
according to \eqref{Eqn8} that $|A|\to0$  as $n \to +\infty$ and $h\to+\infty$. We can divide the term into two parts, by denoting 
     \begin{equation*}
     \begin{split}
         B&=\frac{1}{2}\int_{0}^{T}\int_{D_\Omega}\big(|\widetilde{v}(x,y,t)|^{p-2}\widetilde{v}(x,y,t)-|\widetilde{u}_n(x,y,t)|^{p-2}\widetilde{u}_n(x,y,t)\big)\\
         &\cdot\big(\varphi_1(x,t)-\varphi_1(y,t)) d\sigma dt \\
         &+\frac{1}{2}\int_0^T\int_{D_\Omega}|\widetilde{u}_n(x,y,t)|^{p-2}\widetilde{u}_n(x,y,t)\Big(1-\frac{H'_h(u_n)(x,t)+H'_h(u_n)(y,t)}{2}\Big)\\
         &\cdot\big(\varphi_1(x,t)-\varphi_1(y,t)\big)d\sigma dt\\
         &=B_1+B_2.
     \end{split}
     \end{equation*}
By the Lebesgue dominated convergence theorem and Fatou’s lemma, we get $B_2\to0$, owing to \eqref{Eqn5} we have $f_n\to f$ in $L^1(\Omega_T)$, the right side converges to 0.  And since $0\leq\Theta_{k}(r)\leq k|r|$ as $h\to+\infty$, we have
\begin{align*}
    |\Theta_k\big(v-H_h(u)\big)(T)|&\leq k\big(|v(T)|+|u(T)|\big),\\
    |\Theta_k\big(u_0-H_h(u_0)\big)|&\leq k|u_0|,
\end{align*}
therefore
\begin{align*}
    \int_{\Omega}\Theta_{k}(v-H_h(u))(T) dx&\to\int_{\Omega}\Theta_{k}(v-u)(T) dx,\\
    \int_{\Omega}\Theta_{k}(u_{0}-H_h(u_{0}))dx&\to0.
\end{align*}
Thus, we deduce that
     \begin{equation*}
     \begin{split}
         &\int_\Omega\Theta_k(v-u)(T)dx\\
         &+\frac12\int_0^T\iint_{\{|u|\leq\frac k2,|v|\leq\frac k2\}}\big(|\widetilde{v}(x,y,t)|^{p-2}\widetilde{v}(x,y,t)-|\widetilde{u}(x,y,t)|^{p-2}\widetilde{u}(x,y,t)\big)\\
         &\cdot\big(\widetilde{v}(x,y,t)-\widetilde{u}(x,y,t)\big) d\sigma dt\leq0.
     \end{split}
     \end{equation*}
Since the first term is positive, we conclude that $u = v$ a.e. in $\Omega_T$. 
\end{proof}

\begin{proof}[Proof of Theorem \ref{result3}]
Suppose $u_0, v_0 \in L^2(\Omega)$ and $f, g \in L^{p'}(0, T; X_0^{s(\cdot,\cdot),p}(\Omega)^*)$, then we can obtain $u$ and $v$ are two weak solutions for problems \eqref{Quotient}, and using $\varphi=(u-v)^+\chi_{(0,t)}$ as a test function in Definition \ref{D1}, we get
      \begin{equation*}
      \begin{split}
          &\int_{0}^{t}\int_{\Omega}(u-v)_{t}(u-v)^+ dxd\tau\\
          &+\frac12\int_0^T\iint_{\Omega}\big(|\widetilde{u}(x,y,t)|^{p-2}\widetilde{u}(x,y,t)-|\widetilde{v}(x,y,t)|^{p-2}\widetilde{v}(x,y,t)\big)\\
          &\cdot\big(\varphi(x,t)-\varphi(y,t)\big) d\sigma d\tau \\
          &=\int_0^t\int_\Omega(f-g)(u-v)^+ dxd\tau\leq0.
      \end{split}
      \end{equation*}
Since the second term in the above equation is non-negative, we can obtain
      \begin{equation*}
      \begin{split}
          \int_{0}^{t}\int_{\Omega}(u-v)_{t}(u-v)^+ dxd\tau
          &=\frac12\int_0^t\int_\Omega\frac d{dt}\big[(u-v)^+\big]^2 dxd\tau\\
          &=\frac12\int_\Omega\big[(u-v)^+\big]^2(t) dx-\frac12\int_\Omega\big[(u_0-v_0)^+\big]^2 dx\leq0.
     \end{split}
      \end{equation*}
Due to $u_0 \leq v_0$, the second term is non-negative, we conclude that $(u-v)^+=0$ a.e. in $\Omega_T$,
thus we obtain $u\leq v$ a.e. in $\Omega_T$. 

Now consider $u$ and $v$ as the renormalized (entropy) solution of problems \eqref{Quotient} with $L^1$-data. Assume ${f_n},\;{g_n},\;{u_{0n}}$,\;${v_{0n}}\subset C_0^\infty(\Omega)$ and $f_n\to f,\;g_n\to g$ in $L^1(\Omega_T)$, $u_{0n}\to u_0,\;v_{0n}\to v_0$ in $L^1(\Omega)$ such that
\begin{align*}
f_n &\leq g_n, & 
u_{0n} &\leq v_{0n}, \\
\|f_n\|_{L^1(\Omega_T)} &\leq \|f\|_{L^1(\Omega_T)}, & \|g_n\|_{L^1(\Omega_T)} &\leq \|g\|_{L^1(\Omega_T)}, \\
\|u_{0n}\|_{L^1(\Omega)} &\leq \|u_0\|_{L^1(\Omega)}, & \|v_{0n}\|_{L^1(\Omega_T)} &\leq \|v_0\|_{L^1(\Omega)}.
\end{align*}

We construct two sequences ${u_n}$ and ${v_n}$ for each of the two renormalized (entropy) solutions $u$ and $v$, which can be obtained $u_n\leq v_n$ a.e.in $\Omega_T$ by applying the above conclusions. 
According to Theorem \ref{result1} and Theorem \ref{result2}, it can be concluded that $u_n \to u$ and $v_n \to v$ a.e. in $\Omega_T$. Therefore, we conclude that $u\leq v$ a.e. in $\Omega_T$. 
\end{proof}

\bibliographystyle{plain}
\normalsize
\bibliography{ref}

\end{document}